\documentclass[11pt,reqno]{amsart}
\usepackage{amssymb}
\usepackage[all]{xy}
\usepackage{hyperref}
\usepackage{enumitem}
\usepackage{tikz-cd} 
\usepackage{adjustbox} 
\newtheorem{thm}{Theorem}[section]
\newtheorem{lem}[thm]{Lemma}
\newtheorem{prop}[thm]{Proposition}
\newtheorem{cor}[thm]{Corollary}

\theoremstyle{definition}
\newtheorem{dfn}[thm]{Definition}

\theoremstyle{remark}
\newtheorem{remark}[thm]{Remark}

\newcommand{\CA}{{\mathcal{A}}}

\newcommand{\CF}{{\mathcal{F}}}
\newcommand{\CH}{{\mathcal{H}}}
\newcommand{\CS}{{\mathcal{S}}}
\newcommand{\CI}{{\mathcal{I}}}
\newcommand{\CJ}{{\mathcal{J}}}

\newcommand{\CL}{{\mathcal{L}}}
\newcommand{\CU}{{\mathcal{U}}}
\newcommand{\CT}{{\mathcal{T}}}

\newcommand{\CB}{{\mathcal{B}}}
\newcommand{\CO}{{\mathcal{O}}}

\newcommand{\CR}{{\mathcal{R}}}

\newcommand{\af}{\alpha}
\newcommand{\bt}{\beta}
\newcommand{\gm}{\gamma}

\newcommand{\ld}{\lambda}

\newcommand{\Z}{{\mathbb{Z}}}

\newcommand{\C}{{\mathbb{C}}}
\newcommand{\N}{{\mathbb{N}}}

\newcommand{\reg}{{\operatorname{reg}}}

\begin{document}


\title[ K-Theory and Structural Properties of $C^*$-Algebras Associated with RGBDS ]
{ K-Theory and Structural Properties of $C^*$-Algebras Associated with Relative Generalized Boolean Dynamical Systems }

\author[T. M. Carlsen]{Toke Meier Carlsen}

\address{ Køge, Denmark 
} \email{toke.carlsen\-@\-gmail.\-com}

\author[E. J. Kang]{Eun Ji Kang$^{\dagger}$}
\address{
Research Institute of Mathematics, Seoul National University, Seoul 08826, 
Korea} \email{kkang3333\-@\-gmail.\-com}

\subjclass[2000]{46L05, 46L55, 46L80, 46L85}

\keywords{Boolean Dynamical Systems, K-theory, gauge-invariant ideals, Morita equivalence, $K_0$-liftable, real rank zero}

\thanks{The second author was supported by Basic Science Research Program through the 
National Research Foundation of Korea(NRF) funded by the Ministry of Education(RS-2023-00238961) and the Ministry of Science and ICT(NRF-2022M3H3A1098237).}

\begin{abstract}
We present an explicit formula for the $K$-theory of the $C^*$-algebra associated with a relative generalized Boolean dynamical system $(\CB, \CL, \theta, \CI_\af; \CJ)$. In particular, we find concrete generators for the $K_1$-group of $C^*(\CB, \CL, \theta, \CI_\af; \CJ)$. 
We also prove that every gauge-invariant ideal of $C^*(\CB, \CL, \theta, \CI_\af; \CJ)$ is Morita equivalent to a $C^*$-algebra of a relative generalized Boolean dynamical system.

As a structural application, we show that if the underlying Boolean dynamical system $(\CB, \CL, \theta)$ satisfies Condition (K), then the associated $C^*$-algebra is $K_0$-liftable. Furthermore,  we deduce that if $C^*(\CB, \CL, \theta, \CI_\af; \CJ)$ is separable and purely infinite, then it has real rank zero.

\end{abstract}

\maketitle

\setcounter{equation}{0}
\section{Introduction}

The study of $C^*$-algebras arising from symbolic and combinatorial structures has provided a rich framework for exploring the complex and abstract properties of general $C^*$-algebras. Frameworks such as graph algebras, ultragraph algebras, labeled graph algebras, and higher-rank graph algebras have played a central role in this development. These models offer tractable settings that enable deep insights into the internal structure of the corresponding algebras.

In this context, the \emph{relative generalized Boolean dynamical systems}, introduced in \cite{CaK2}, provide a unifying and flexible generalization that encompasses many previously studied constructions. These systems add ideal data to the underlying Boolean dynamics, thereby broadening the class of $C^*$-algebras accessible to combinatorial methods.  A fundamental problem in the study of such algebras is the computation of their $K$-groups, which serve as powerful invariants for classification and structural analysis.

In this paper, we investigate the $K$-theoretic and structural properties of $C^*$-algebras associated with relative generalized Boolean dynamical systems. 
The $C^*$-algebra associated with a relative generalized Boolean dynamical system was introduced in \cite{CaK2} as a relative Cuntz--Pimsner algebra arising from a $C^*$-correspondence. 
Building on the ideas in \cite{BaCaPa2017} and utilizing Katsura's general framework for Cuntz--Pimsner algebras developed in \cite{Ka2004c}, we compute the $K$-groups of our algebras as the kernel and cokernel of a certain map induced by the associated $C^*$-correspondence.

Beyond identifying the $K$-groups, a central goal of this paper is to construct an explicit group isomorphism between the kernel of the map and the $K_1$-group of the $C^*$-algebra. While the $K_0$-group admits a canonical set of generators, describing generators for the $K_1$-group requires a more delicate and involved analysis. Our construction is inspired by the techniques developed in \cite{CET}, but the added complexity of relative generalized Boolean dynamical systems demands a substantially more intricate approach.

This paper also develops a detailed analysis of the ideal structure of $C^*$-algebras associated with relative generalized Boolean dynamical systems, focusing particularly on gauge-invariant ideals and their classification up to Morita equivalence. A central theme of this analysis is to understand how such gauge-invariant ideals reflect the internal combinatorial and dynamical structure of the system.

To this end, we introduce a canonical extension procedure that assigns to each relative generalized Boolean dynamical system $(\CB, \CL, \theta, \CI_\alpha; \CJ)$ an extended system $(\tilde{\CB}, \CL, \tilde{\theta}, \tilde{\CI}_\alpha; \tilde{\CJ})$, together with a natural injective map $\iota:\CB \hookrightarrow \tilde{\CB}$.
 In Theorem~\ref{prop:one}, we show that this extension induces a $*$-isomorphism between the corresponding $C^*$-algebras, allowing us to replace the original system with its extension without loss of information.
 
  This framework becomes especially powerful in the study of gauge-invariant ideals. Given such an ideal, we construct a suitable subsystem of the extended Boolean dynamical system whose associated $C^*$-algebra naturally embeds as a full hereditary subalgebra of the ideal. This construction, detailed in Theorem~\ref{gii}, leads to a Morita equivalence, allowing ideal-theoretic questions to be reduced to more tractable subsystems where existing structural results can be applied effectively.

Taking advantage of this structural insight,  we employ Morita equivalences alongside the established $K$-theoretic framework to prove that when the underlying Boolean dynamical system satisfies Condition~(K), the associated $C^*$-algebra is $K_0$-liftable. This property is important as it allows the lifting of projections in $K_0$, thereby facilitating finer classification and structural results.
Notably, as a direct consequence of these developments, we establish that every separable, purely infinite $C^*$-algebra within this class has real rank zero. This stands in contrast to the general landscape of purely infinite $C^*$-algebras, where the real rank can often be greater than zero. The real rank zero property implies a highly regular and well-behaved structure, which has numerous consequences for classification theory and the understanding of these algebras.

The paper is organized as follows. In Section~\hyperref[preliminary]{2}, we provide a comprehensive review of the necessary background on relative generalized Boolean dynamical systems and their associated \(C^*\)-algebras. Section~\hyperref[K-theory]{3} focuses on computing the 
$K$-groups of these algebras and providing an explicit isomorphism for the 
$K_1$-group.
 In Section~\hyperref[Morita]{4}, we rigorously define a canonical extension of relative generalized Boolean dynamical systems and introduce appropriate subsystems corresponding to gauge-invariant ideals. We then prove that the associated $C^*$-algebras of these subsystems embed as full hereditary subalgebras of the ideals, establishing Morita equivalences.
Finally, Section~\hyperref[$K_0$-liftability]{5} establishes the \(K_0\)-liftability property under Condition~(K) and describes related structural consequences.

\section{Preliminaries}\label{preliminary}

\subsection{Boolean algebras and Boolean dynamical systems}

A lattice $\mathcal{B}$ is said to be a \textit{Boolean algebra} if it is distributive, relatively complemented, and possesses a least element. For elements $A, B \in \mathcal{B}$, the meet (or intersection) and join (or union) are denoted by $A \cap B$ and $A \cup B$, respectively. The relative complement of $A$ in $B$ is written as $B \setminus A$, and the least element of $\mathcal{B}$ is denoted by $\emptyset$.
The symbol \( \subseteq \) denotes the partial order  on \(\mathcal{B}\)
and say that $A$ is a \emph{subset} of $B$ whenever $A \subseteq B$.

A nonempty subset \(\mathcal{I} \subseteq \mathcal{B}\) is called an \emph{ideal} if it is closed under finite unions, that is, for any \(A, B \in \mathcal{I}\), the union \(A \cup B\) belongs to \(\mathcal{I}\), and if it is downward closed: whenever \(A \in \mathcal{I}\) and \(B \subseteq A\), then \(B \in \mathcal{I}\) as well. Given \(A \in \mathcal{B}\), we define the ideal \(\mathcal{I}_A\) generated by \(A\) as
$
\mathcal{I}_A := \{ B \in \mathcal{B} : B \subseteq A \}.
$

When \(\mathcal{I}\) is an ideal in the Boolean algebra \(\mathcal{B}\), we define a relation \(\sim\) on \(\mathcal{B}\) by
$
A \sim B \quad \text{if and only if} \quad A \cup A' = B \cup B' \text{ for some } A', B' \in \mathcal{I}.
$
This relation is an equivalence relation on \(\mathcal{B}\). The equivalence class of an element \(A \in \mathcal{B}\) with respect to \(\mathcal{I}\) is denoted by \([A]_\mathcal{I}\), or simply \([A]\) when the ideal is clear from context. The quotient set
$
\mathcal{B} / \mathcal{I} := \{ [A]_\mathcal{I} : A \in \mathcal{B} \}
$
inherits a Boolean algebra structure via the operations
\[
[A]_\mathcal{I} \cap [B]_\mathcal{I} := [A \cap B]_\mathcal{I}, \quad
[A]_\mathcal{I} \cup [B]_\mathcal{I} := [A \cup B]_\mathcal{I}, \quad
[A]_\mathcal{I} \setminus [B]_\mathcal{I} := [A \setminus B]_\mathcal{I}.
\]
A nonempty subset \(\xi \subseteq \mathcal{B}\) is said to be a \emph{filter} if it 
satisfies the following conditions: \(\emptyset \notin \xi\), it is closed under finite intersections, that is, \(A \cap B \in \xi\) whenever \(A, B \in \xi\), and it is upward closed in the sense that whenever \(A \in \xi\) and \(A \subseteq B\), then \(B \in \xi\).
 A filter is called  an \emph{ultrafilter} if it is maximal with respect to inclusion among all filters. A \emph{prime filter} is a filter with the additional property that whenever \(A \cup B \in \xi\), then either \(A \in \xi\) or \(B \in \xi\). In a Boolean algebra \(\mathcal{B}\), a filter is an ultrafilter if and only if it is prime.
We denote by \(\widehat{\mathcal{B}}\) the set of all ultrafilters on \(\mathcal{B}\). For each \(A \in \mathcal{B}\), define
\[
Z(A) := \{ \xi \in \widehat{\mathcal{B}} \mid A \in \xi \}.
\]
The family \(\{ Z(A) : A \in \mathcal{B} \}\) forms a basis for a topology on \(\widehat{\mathcal{B}}\). Equipped with this topology, \(\widehat{\mathcal{B}}\) becomes a totally disconnected, locally compact Hausdorff space. Moreover, each set \(Z(A)\) is both compact and open.

 Let \(\mathcal{B}, \mathcal{B}'\) be Boolean algebras. A map \(\phi: \mathcal{B} \to \mathcal{B}'\) is called a \emph{Boolean homomorphism} if for all \(A, B \in \mathcal{B}\), it satisfies
$$
\phi(A \cap B) = \phi(A) \cap \phi(B), \quad
\phi(A \cup B) = \phi(A) \cup \phi(B), \quad ~\text{and}~
\phi(A \setminus B) = \phi(A) \setminus \phi(B).
$$
A Boolean homomorphism \(\theta: \mathcal{B} \to \mathcal{B}\) satisfying \(\theta(\emptyset) = \emptyset\) is called an \emph{action} on \(\mathcal{B}\).

Given a set \(\mathcal{L}\), we define finite words over \(\mathcal{L}\) as follows: for each \(n \in \mathbb{N}\), let \(\mathcal{L}^n := \{ \alpha_1 \cdots \alpha_n \mid \alpha_i \in \mathcal{L} \}\), and define \(\mathcal{L}^0 := \{ \emptyset \}\), \(\mathcal{L}^* := \bigcup_{n \geq 0} \mathcal{L}^n\). For any \(\alpha \in \mathcal{L}^n\), we write \(|\alpha| := n\).
The concatenation of two words \(\alpha = \alpha_1 \cdots \alpha_n\) and \(\beta = \beta_1 \cdots \beta_m\) in \(\mathcal{L}^*\) is denoted by \(\alpha\beta := \alpha_1 \cdots \alpha_n \beta_1 \cdots \beta_m\). By convention, concatenation with the empty word \(\emptyset\) leaves the other word unchanged. For \(\alpha \in \mathcal{L}^*\) and \(k \in \mathbb{N}\), we define \(\alpha^k\) to be the \(k\)-fold concatenation of \(\alpha\) with itself, and set \(\alpha^0 := \emptyset\).
For a word \(\alpha = \alpha_1 \cdots \alpha_n\), the subword from position \(i\) to \(j\) (\(1 \leq i \leq j \leq n\)) is written as \(\alpha_{i,j} := \alpha_i \cdots \alpha_j\), with \(\alpha_{i,i} := \alpha_i\).

A \emph{Boolean dynamical system} is a triple \((\mathcal{B}, \mathcal{L}, \theta)\) where \(\mathcal{B}\) is a Boolean algebra, \(\mathcal{L}\) is a set, and \(\theta = (\theta_\alpha)_{\alpha \in \mathcal{L}}\) is a family of actions on \(\mathcal{B}\). For any nonempty word \(\beta = \beta_1 \cdots \beta_n \in \mathcal{L}^*\), we define the composite action
$
\theta_\beta := \theta_{\beta_n} \circ \cdots \circ \theta_{\beta_1},
$
and let \(\theta_{\emptyset} := \mathrm{Id}\), the identity map.

Given \(B \in \mathcal{B}\), define
$
\Delta_B^{(\CB, \CL,\theta)} := \{ \alpha \in \mathcal{L} \mid \theta_\alpha(B) \neq \emptyset \}.
$
 We say that a set \(A \in \mathcal{B}\) is \emph{regular} if for every nonempty subset \(B \subseteq A\), we have \(0 < |\Delta_B^{(\CB, \CL,\theta)}| < \infty\).
The collection of regular elements is denoted by \(\mathcal{B}_{\mathrm{reg}}\), which forms an ideal in \(\mathcal{B}\).

A Boolean dynamical system \((\mathcal{B}, \mathcal{L}, \theta)\) is said to satisfy \emph{Condition (K)} if there is no triple \(( (\beta, \eta), A )\), where \((\beta, \eta)\) is an \emph{ultrafilter cycle}, that is, \(\beta \in \mathcal{L}^* \setminus \{\emptyset\}\), \(\eta \in \widehat{\mathcal{B}}\), and \(\theta_\beta(A) \in \eta\) for all \(A \in \eta\),  such that the following holds:  
for every \(\gamma \in \mathcal{L}^* \setminus \{\emptyset\}\) and \(B \in \mathcal{I}_A\), if \(\theta_\gamma(B) \in \eta\), then \(B \in \eta\) and \(\gamma = \beta^k\) for some \(k \in \mathbb{N}\).

\subsection{Relative generalized Boolean dynamical systems and their $C^*$-algebras}

A \emph{generalized Boolean dynamical system} is a quadruple \((\CB, \CL, \theta, \CI_\alpha)\), where \((\CB, \CL, \theta)\) is a Boolean dynamical system, and \(\{\CI_\alpha\}_{\alpha \in \CL}\) is a family of ideals in \(\CB\) satisfying
\[
\mathcal{R}_\alpha^{(\CB, \CL,\theta)} \subseteq \CI_\alpha \quad \text{for all } \alpha \in \CL,
\]
where
\[
\mathcal{R}_\alpha^{(\CB, \CL,\theta)} := \{ A \in \CB : A \subseteq \theta_\alpha(B) \text{ for some } B \in \CB \}.
\]
A \emph{relative generalized Boolean dynamical system} is a quintuple \((\CB, \CL, \theta, \CI_\alpha; \CJ)\), where \((\CB, \CL, \theta, \CI_\alpha)\) is a generalized Boolean dynamical system and \(\CJ\) is an ideal of the regular part \(\CB_{\reg}\).
Similarly, a \emph{relative Boolean dynamical system} is a quadruple \((\CB, \CL, \theta; \CJ)\), where \((\CB, \CL, \theta)\) is a Boolean dynamical system and \(\CJ\) is an ideal of \(\CB_{\reg}\).

\begin{dfn}\label{def:representation of RGBDS}(\cite[Definition 3.3]{CaK2})
Let $(\CB,\CL,\theta, \CI_\alpha; \CJ)$ be a relative generalized Boolean dynamical system. A {\em  $(\CB, \CL, \theta, \CI_\alpha;\CJ)$-representation} consists of a family of projections $\{P_A:A\in\mathcal{B}\}$ and a family of partial isometries $\{S_{\alpha,B}:\alpha\in\mathcal{L},\ B\in\mathcal{I}_\alpha\}$ in a $C^*$-algebra such that for $A,A'\in\mathcal{B}$, $\alpha,\alpha'\in\mathcal{L}$, $B\in\mathcal{I}_\alpha$ and $B'\in\mathcal{I}_{\alpha'}$,
\begin{enumerate}
\item[(i)] $P_\emptyset=0$, $P_{A\cap A'}=P_AP_{A'}$, and $P_{A\cup A'}=P_A+P_{A'}-P_{A\cap A'}$;
\item[(ii)] $P_AS_{\alpha,B}=S_{\alpha,  B}P_{\theta_\alpha(A)}$;
\item[(iii)] $S_{\alpha,B}^*S_{\alpha',B'}=\delta_{\alpha,\alpha'}P_{B\cap B'}$;
\item[(iv)] $P_A=\sum_{\alpha \in\Delta_A}S_{\alpha,\theta_\alpha(A)}S_{\alpha,\theta_\alpha(A)}^*$ for all  $A\in \mathcal{J}$. 
\end{enumerate}
\end{dfn}

The existence of a universal \((\mathcal{B}, \mathcal{L}, \theta, \mathcal{I}_\alpha; \mathcal{J})\)-representation
$$
\{p_A, s_{\alpha, B} : A \in \mathcal{B},~ \alpha \in \mathcal{L},~ B \in \mathcal{I}_\alpha\}
$$
is established in \cite[Theorem 5.5]{CaK2}. 
We denote by \( C^*(\mathcal{B}, \mathcal{L}, \theta, \mathcal{I}_\alpha; \mathcal{J}) \) the $C^*$-algebra generated by this universal representation.

When \(\mathcal{J} = \emptyset\), such a representation is referred to as a \textit{Toeplitz representation} of \((\mathcal{B}, \mathcal{L}, \theta, \mathcal{I}_\alpha)\), and the corresponding $C^*$-algebra
\[
\mathcal{T}(\mathcal{B}, \mathcal{L}, \theta, \mathcal{I}_\alpha) := C^*(\mathcal{B}, \mathcal{L}, \theta, \mathcal{I}_\alpha; \emptyset)
\]
is called the \textit{Toeplitz $C^*$-algebra} of \((\mathcal{B}, \mathcal{L}, \theta, \mathcal{I}_\alpha)\).

In the case where \(\mathcal{J} = \mathcal{B}_{\mathrm{reg}}\), we simply write
\[
C^*(\mathcal{B}, \mathcal{L}, \theta, \mathcal{I}_\alpha) := C^*(\mathcal{B}, \mathcal{L}, \theta, \mathcal{I}_\alpha; \mathcal{B}_{\mathrm{reg}})
\]
and refer to it as the \textit{$C^*$-algebra of \((\mathcal{B}, \mathcal{L}, \theta, \mathcal{I}_\alpha)\)}.

The universal $C^*$-algebra \(C^*(\mathcal{B}, \mathcal{L}, \theta, \mathcal{I}_\alpha; \mathcal{J})\) admits a strongly continuous action \(\gamma : \mathbb{T} \to \mathrm{Aut}(C^*(\mathcal{B}, \mathcal{L}, \theta, \mathcal{I}_\alpha; \mathcal{J}))\), known as the \emph{gauge action}, characterized by
\[
\gamma_z(p_A) = p_A, \qquad \gamma_z(s_{\alpha, B}) = z s_{\alpha, B},
\]
for all \(z \in \mathbb{T}\), \(A \in \mathcal{B}\), \(\alpha \in \mathcal{L}\), and \(B \in \mathcal{I}_\alpha\).

An ideal \(I \subseteq C^*(\mathcal{B}, \mathcal{L}, \theta, \mathcal{I}_\alpha; \mathcal{J})\) is said to be \emph{gauge-invariant} if \(\gamma_z(I) = I\) for all \(z \in \mathbb{T}\).
It was shown in \cite[Theorem 3.1]{CaK3} that the Boolean dynamical system \((\mathcal{B}, \mathcal{L}, \theta)\) satisfies Condition (K) if and only if every ideal of \(C^*(\mathcal{B}, \mathcal{L}, \theta, \mathcal{I}_\alpha; \mathcal{J})\) is gauge-invariant.

\begin{remark}

Let $
\{P_A, S_{\alpha,B} : A \in \mathcal{B},\ \alpha \in \mathcal{L},\ B \in \mathcal{I}_\alpha\}
$ be a  \((\mathcal{B}, \mathcal{L}, \theta, \mathcal{I}_\alpha; \mathcal{J})\)-representation. 
For each word \(\alpha = \alpha_1 \alpha_2 \cdots \alpha_n \in \mathcal{L}^*\), we define \(\mathcal{I}_\alpha := \mathcal{B}\) if \(\alpha = \emptyset\), and otherwise set
\[
\mathcal{I}_\alpha := \left\{ A \in \mathcal{B} : A \subseteq \theta_{\alpha_2 \cdots \alpha_n}(B) \text{ for some } B \in \mathcal{I}_{\alpha_1} \right\}.
\]
Then, for \(\alpha = \alpha_1 \alpha_2 \cdots \alpha_n \in \mathcal{L}^*\) and \(A \in \mathcal{I}_\alpha\), we define recursively
\[
S_{\alpha, A} := 
\begin{cases}
P_A & \text{if } \alpha = \emptyset, \\
S_{\alpha_1, B} S_{\alpha_2, \theta_{\alpha_2}(B)} \cdots S_{\alpha_n, A} & \text{otherwise},
\end{cases}
\]
where \(B \in \mathcal{I}_{\alpha_1}\) is such that \(A \subseteq \theta_{\alpha_2 \cdots \alpha_n}(B)\).
It then follows from \cite[Remark 3.11]{CaK2} that the $C^*$-algebra generated by the \((\mathcal{B}, \mathcal{L}, \theta, \mathcal{I}_\alpha; \mathcal{J})\)-representation is given by
\[
C^*(P_A, S_{\alpha, B}) := \overline{\operatorname{span}} \left\{ S_{\mu, A} S_{\nu, A}^* : \mu, \nu \in \mathcal{L}^*,\ A \in \mathcal{I}_\mu \cap \mathcal{I}_\nu \right\}.
\]
\end{remark}


\subsection{Gauge-invariant ideals}
Let \((\mathcal{B}, \mathcal{L}, \theta; \mathcal{J})\) be a relative Boolean dynamical system. An ideal \(\mathcal{H} \subseteq \mathcal{B}\)   is called \emph{hereditary} if \(\theta_\alpha(A) \in \mathcal{H}\) for all \(A \in \mathcal{H}\) and \(\alpha \in \mathcal{L}\). It is \emph{\(\mathcal{J}\)-saturated} (or simply \emph{saturated} when \(\mathcal{J} = \mathcal{B}_{\mathrm{reg}}\)) if for every \(A \in \mathcal{J}\), the condition \(\theta_\alpha(A) \in \mathcal{H}\) for all \(\alpha \in \Delta_A\) implies \(A \in \mathcal{H}\).

Given that \(\mathcal{H}\) is a hereditary and \(\mathcal{J}\)-saturated ideal of \(\mathcal{B}\), the maps \(\theta_\alpha\) induce well-defined maps on the quotient Boolean algebra \(\mathcal{B} / \mathcal{H}\) by
\[
\theta_\alpha([A]_{\mathcal{H}}) := [\theta_\alpha(A)]_{\mathcal{H}}.
\]
This gives rise to a quotinet Boolean dynamical system \((\mathcal{B} / \mathcal{H}, \mathcal{L}, \theta)\). The set
\[
\mathcal{B}_{\mathcal{H}} := \{ A \in \mathcal{B} \mid [A]_{\mathcal{H}} \in (\mathcal{B}/\mathcal{H})_{\mathrm{reg}} \}
\]
defines an ideal in \(\mathcal{B}\) that contains both \(\mathcal{H}\) and \(\mathcal{J}\).

Let $(\CB, \CL, \theta, \CI_\alpha; \CJ)$ be a relative generalized Boolean dynamical system. Suppose $\CH$ is a hereditary and $\CJ$-saturated ideal of $\CB$, and let $\CS$ be an ideal in $\CB_\CH$ satisfying $\CH \cup \CJ \subseteq \CS$. We define the ideal $I_{(\CH, \CS)} \subseteq C^*(\CB, \CL, \theta, \CI_\alpha; \CJ)$ as the one generated by
\[
\bigl\{ p_A - \sum_{\alpha \in \Delta_{[A]_\CH}} s_{\alpha, \theta_\alpha(A)} s_{\alpha, \theta_\alpha(A)}^* : A \in \CS \bigr\}.
\]

Given any ideal $I$ in $C^*(\CB, \CL, \theta, \CI_\alpha; \CJ)$, we associate two subsets of $\CB$:
\[
\CH_I := \{ A \in \CB : p_A \in I \}, \quad 
\CS_I := \left\{ A \in \CB_{\CH_I} : p_A - \sum_{\alpha \in \Delta_{[A]_{\CH_I}}} s_{\alpha, \theta_\alpha(A)} s_{\alpha, \theta_\alpha(A)}^* \in I \right\}.
\]
It follows that $\CH_I$ is hereditary and $\CJ$-saturated, and $\CS_I$ is an ideal in $\CB_{\CH_I}$ satisfying $\CH_I \cup \CJ \subseteq \CS_I$. Furthermore, we always have $I_{(\CH_I, \CS_I)} \subseteq I$, and equality holds if and only if $I$ is gauge-invariant.

The assignment $(\CH, \CS) \mapsto I_{(\CH, \CS)}$ yields a lattice isomorphism between:
\begin{itemize}
    \item the lattice of pairs $(\CH, \CS)$, where $\CH$ is a hereditary $\CJ$-saturated ideal of $\CB$ and $\CS$ is an ideal in $\CB_\CH$ such that $\CH \cup \CJ \subseteq \CS$, with the partial order $(\CH_1, \CS_1) \leq (\CH_2, \CS_2)$ if $\CH_1 \subseteq \CH_2$ and $\CS_1 \subseteq \CS_2$, and
    \item the lattice of gauge-invariant ideals of $C^*(\CB, \CL, \theta, \CI_\alpha; \CJ)$.
\end{itemize}

Additionally, for every such pair $(\CH, \CS)$, there exists a $*$-isomorphism
\[
\phi: C^*(\CB / \CH, \CL, \theta, [\CI_\alpha]; [\CS]) \to C^*(\CB, \CL, \theta, \CI_\alpha; \CJ) / I_{(\CH, \CS)}
\]
satisfying $\phi(p_{[A]}) = p_A + I_{(\CH, \CS)}$ for all $A \in \CB$ and $\phi(s_{\alpha, [B]}) = s_{\alpha, B} + I_{(\CH, \CS)}$ for $\alpha \in \CL$, $B \in \CI_\alpha$, where \([\mathcal{I}_\alpha] := \{ [A]_{\mathcal{H}} \mid A \in \mathcal{I}_\alpha \}\) for each \(\alpha \in \mathcal{L}\) and \([\mathcal{S}] := \{ [A]_{\mathcal{H}} \mid A \in \mathcal{S} \}\).



\section{$K$-Theory of $C^*$-algebras from Boolean Dynamical Systems}\label{K-theory}

The $C^*$-algebra of a relative generalized Boolean dynamical system can be realized as a relative Cuntz--Pimsner algebra of a $C^*$-correspondence, and the $C^*$-algebra of a generalized Boolean dynamical system can be realized as a Cuntz--Pimsner algebra of the $C^*$-correspondence. 
As in \cite{BaCaPa2017}, we shall use \cite[Theorem 8.6 and Proposition 8.7]{Ka2004c} to compute the $K$-theory of our $C^*$-algebras. 
We begin by recalling the construction of the 
$C^*$-correspondence of  a relative generalized Boolean dynamical system given in \cite[Section 5]{CaK2}.

Let $(\mathcal{B}, \mathcal{L}, \theta, \mathcal{I}_\alpha; \mathcal{J})$ be a relative generalized Boolean dynamical system.  
For each \(A \in \mathcal{B}\), define \[\chi_A := 1_{Z(A)} \in C_0(\widehat{\mathcal{B}}).\] By the Stone--Weierstrass theorem, the $C^*$-algebra \(\mathcal{A}(\mathcal{B}, \mathcal{L}, \theta)\), defined as
\[
\mathcal{A}(\mathcal{B}, \mathcal{L}, \theta) := \overline{\operatorname{span}} \{ \chi_A : A \in \mathcal{B} \},
\]
coincides with \(C_0(\widehat{\mathcal{B}})\). For notational convenience, we will simply write
$
\mathcal{A} := \mathcal{A}(\mathcal{B}, \mathcal{L}, \theta) = C_0(\widehat{\mathcal{B}}).
$
For each $\alpha \in \mathcal{L}$, define the ideal
\[
X_\alpha := \overline{\operatorname{span}} \{ \chi_A : A \in \mathcal{I}_\alpha \} \subseteq \mathcal{A}.
\]
Then, $X_\alpha$ becomes a right Hilbert $\mathcal{A}$-module with right action given by the usual multiplication in $\mathcal{A}$ and inner product
\[
\langle x, y \rangle := x^* y \quad \text{for } x, y \in X_\alpha.
\]
Define the right Hilbert $\mathcal{A}$-module
\[
X(\mathcal{B}, \mathcal{L}, \theta, \mathcal{I}_\alpha) := \bigoplus_{\alpha \in \mathcal{L}} X_\alpha
\]
with inner product and right action given by
\[
\left\langle (x_\alpha)_{\alpha \in \mathcal{L}}, (y_\alpha)_{\alpha \in \mathcal{L}} \right\rangle := \sum_{\alpha \in \mathcal{L}} x_\alpha^* y_\alpha, \quad
(x_\alpha)_{\alpha \in \mathcal{L}} \cdot f := (x_\alpha f)_{\alpha \in \mathcal{L}}
\]
for $(x_\alpha)_{\alpha \in \mathcal{L}}, (y_\alpha)_{\alpha \in \mathcal{L}} \in X(\mathcal{B}, \mathcal{L}, \theta, \mathcal{I}_\alpha)$ and $f \in \mathcal{A}$.
It is straightforward to verify that
\begin{align*}
X(\mathcal{B}, \mathcal{L}, \theta, \mathcal{I}_\alpha) 
&= \overline{\operatorname{span}} \{ e_{\alpha, A} : \alpha \in \mathcal{L},\ A \in \mathcal{I}_\alpha \} \\
&= \overline{\operatorname{span}} \{ e_{\alpha, A} : \alpha \in \mathcal{L},\ A \in F\ \text{for some } F \in \mathcal{F} \},
\end{align*}
where each $e_{\alpha, A} := (\delta_{\alpha, \beta} \chi_A)_{\beta \in \mathcal{L}} \in \bigoplus_{\alpha \in \mathcal{L}} X_\alpha$, and $\mathcal{F}$ denotes the collection of finite families of mutually disjoint elements in $\mathcal{I}_\alpha$.

We then have a unique $*$-homomorphism $\phi_\af: \CA \to X_\af$ defined by $$\phi_\af(\chi_A)=\chi_{\theta_\af(A)}$$ for every $A \in \CB$ and for each $\af \in \CL$ (see \cite[Lemma 3.4]{BaCaPa2017}). This induces a $*$-homomorphism
\[
\phi : \mathcal{A} \to \mathfrak{L}\big(X(\mathcal{B}, \mathcal{L}, \theta, \mathcal{I}_\alpha)\big)
\]
defined by
\[
\phi(f)\big((x_\alpha)_{\alpha \in \mathcal{L}}\big) := (\phi_\alpha(f) x_\alpha)_{\alpha \in \mathcal{L}} \quad \text{for } f \in \mathcal{A},
\]
 where $ \mathfrak{L}(X(\CB,\CL,\theta, \CI_\af))$ denotes the $C^*$-algebra of adjointable operators on $X(\CB,\CL,\theta, \CI_\af)$ (see \cite[Lemma 3.6]{BaCaPa2017}).   Thus, $X:=X(\CB,\CL,\theta, \CI_\af)$ is a $C^*$-correspondence over $\CA$.

Let $\mathfrak{K}(X)$ denote the ideal of generalized compact operators on $X$, that is,
\begin{align*}\label{compact} 
\mathfrak{K}(X)=\overline{\operatorname{\operatorname{span}}}\{\Theta_{x,y} \in\mathfrak{L}(X):x,y \in X \},
\end{align*}
where $\Theta_{x,y}$ is the rank-one operator defined by $\Theta_{x,y}(z) := x \langle y, z \rangle$.
 We  define an ideal $J(X)$ of $\CA$ by
$
J(X):=\phi^{-1}(\mathfrak{K}(X)).
$

Let
\[
K_\mathcal{J} := \overline{\operatorname{span}} \{ \chi_A : A \in \mathcal{J} \}
\]
be an ideal of $J(X)$. We denote by
$
\mathcal{O}_{(K_\mathcal{J}, X)}
$
the $C^*$-algebra generated by a universal $K_\mathcal{J}$-coisometric representation of \( X \) (see \cite[page 15]{CaK2}).  
It follows from \cite[Theorem 5.5]{CaK2} that
\[
C^*(\mathcal{B}, \mathcal{L}, \theta, \mathcal{I}_\alpha; \mathcal{J}) \cong \mathcal{O}_{(K_\mathcal{J}, X)}.
\]

\medskip

We also define the ideal
\[
J_{X}
:= \phi^{-1}(\mathfrak{K}(X)) \cap \ker(\phi)^\perp,
\]
which is known to satisfy
\[
J_{X}
= \overline{\operatorname{span}} \{ \chi_A : A \in \mathcal{B}_{\mathrm{reg}} \}
\quad \text{(see \cite[Lemma 5.2(2)]{CaK2})}.
\]
In this case, we have the isomorphism
\[
C^*(\mathcal{B}, \mathcal{L}, \theta, \mathcal{I}_\alpha)
\cong \mathcal{O}_{X}
\quad \text{(see \cite[Corollary 5.6(2)]{CaK2})}.
\]

\subsection{The K-theory of $C^*(\CB,\CL,\theta,\CI_\alpha)$}
To compute the $K$-groups of the $C^*$-algebra associated with a generalized Boolean dynamical system via \cite[Theorem 8.6]{Ka2004c}, we first analyze the map
\[
[X]: K_*\big(J_{X}\big) \to K_*(\mathcal{A}).
\]
We begin by recalling its definition and key properties in the context of the $C^*$-correspondence \( X (= X(\mathcal{B}, \mathcal{L}, \theta, \mathcal{I}_\alpha)) \).
Let
\[
D_X := \mathfrak{K}(X \oplus \mathcal{A})
\]
be the linking algebra of \( X \), as defined in \cite[Definition B.1]{Ka2004c}.
We denote by
\[
\iota_{\mathcal{A}}: \mathcal{A} \hookrightarrow D_X \quad \text{and} \quad \iota_{\mathfrak{K}(X)}: \mathfrak{K}(X) \hookrightarrow D_X
\]
the canonical embeddings.
By \cite[Proposition B.3]{Ka2004c}, the induced map
\[
(\iota_{\mathcal{A}})_*: K_*(\mathcal{A}) \to K_*(D_X)
\]
is an isomorphism.
Then the map
\[
[X]:= (\iota_{\mathcal{A}})_*^{-1} \circ (\iota_{\mathfrak{K}(X)})_* \circ (\phi|_{J_X})_*
\]
is the one induced by the $C^*$-correspondence \( X \), as described in \cite[Definition 8.3 and Remark B.4]{Ka2004c}.

\begin{lem}\label{equiv. in K(D)} Let $(\CB,\CL,\theta, \CI_\af)$ be a generalized Boolean dynamical system,  $\af \in \CL$ and $A \in \CI_\af$. Then for $[\chi_A]_0 \in K_0(\CA)$ and $[\Theta_{e_{\af,A}, e_{\af,A}}]_0 \in K_0(\mathfrak{K}(X))$, we have 
\[(\iota_{\CA})_* ([\chi_A]_0)     =(\iota_{\mathfrak{K}(X)})_*  ([\Theta_{e_{\af,A}, e_{\af,A}}]_0) \]
in $K_0(D_{X})$.
\end{lem}

\begin{proof}
Let $v: X \times \CA \to  X \times \CA$ be a map given by $ (\eta, f) \mapsto (e_{\af,A}f,0)$, where $e_{\af,A}=(\delta_{\af,\bt}\chi_A)_{\bt \in \CL}$.  
It is straightforward to check that $v \in \mathfrak{K}(X \oplus \CA)$, and that $vv^*=\iota_{\mathfrak{K}(X)}(\Theta_{e_{\af, A},e_{\af, A}})$ and $v^*v=\iota_{\CA}(\chi_A)$. It means that $\iota_{\mathfrak{K}(X)}(\Theta_{e_{\af, A},e_{\af, A}})$ and $\iota_{\CA}(\chi_A)$ are equivalent in $K_0(D_{X})$. So we are done. 
\end{proof}

Since
\[
\mathcal{A} = \overline{\operatorname{span}}\{\chi_A : A \in \mathcal{B}\} ~\text{and}~ 
J_X = \overline{\operatorname{span}}\{\chi_A : A \in \mathcal{B}_{\mathrm{reg}}\},
\]
we have natural isomorphisms
\[
K_0(\mathcal{A}) \cong \operatorname{span}_\mathbb{Z}\{\chi_A : A \in \mathcal{B}\} ~\text{and}~ 
K_0(J_X) \cong \operatorname{span}_\mathbb{Z}\{\chi_A : A \in \mathcal{B}_{\mathrm{reg}}\},
\]
where each characteristic function $\chi_A$ corresponds to the class $[\chi_A]_0$ in the respective $K_0$-group.
For simplicity, we identify $K_0(\mathcal{A})$ with $\operatorname{span}_\mathbb{Z}\{\chi_A : A \in \mathcal{B}\}$ and $K_0(J_X)$ with $\operatorname{span}_\mathbb{Z}\{\chi_A : A \in \mathcal{B}_{\mathrm{reg}}\}$, and write $[\chi_A]_0$ simply as $\chi_A$.

\begin{lem}\label{map Phi} Identifying $K_0(\CA)$ with $\operatorname{span}_\Z\{\chi_A: A \in \CB\}$ and   $K_0(J_{X})$ with $\operatorname{span}_\Z\{\chi_A: A \in \CB_{reg}\}$, the homomorphism $[X]: K_0(J_{X}) \to K_0(\CA)$ induces the map $\Phi: \operatorname{span}_\Z\{\chi_A: A \in \CB_{reg}\} \to \operatorname{span}_\Z\{\chi_A : A \in \CB\}$  determined by 
$$\chi_A \mapsto \sum_{\af \in \Delta_A} \chi_{\theta_\af(A)}~~ \text{for}~ A \in \CB_{reg}.$$
 \end{lem}

\begin{proof} Let $A \in \CB_{reg}$. Since $\phi_\af(\chi_A)=\chi_{\theta_\af(A)}$, we have $\phi(\chi_A)= \sum_{\af \in \Delta_A} \Theta_{e_{\af,\theta_\af(A)}, e_{\af, \theta_\af(A)}}$.  It then follows from  Lemma \ref{equiv. in K(D)} that 
\begin{align*} (\iota_{\mathfrak{K}(X)})_* \circ (\phi|_{J_{X}})_* ([\chi_A]_0)&= (\iota_{\mathfrak{K}(X)})_*\Big( \Big[ \sum_{\af \in \Delta_A}\Theta_{e_{\af,\theta_\af(A)},e_{\af,\theta_\af(A)} }\Big]_0\Big) \\
&= \sum_{\af \in \Delta_A}  (\iota_{\mathfrak{K}(X)})_*([\Theta_{e_{\af,\theta_\af(A)},e_{\af,\theta_\af(A)}}]_0) \\
&= \sum_{\af \in \Delta_A}  (\iota_{\CA})_*([\chi_{\theta_\af(A)}]_0).
\end{align*}
Then our result follows from the given identifications and the definition of $[X]$. 
\end{proof}

\begin{lem}\label{K1-zero} Let  $(\CB,\CL,\theta, \CI_\af)$ be a generalized Boolean dynamical system. Then $\CA$ and $J_{X}$ are locally finite dimensional algebras with  $K_1(\CA)=K_1(J_{X})=0$.
\end{lem}

\begin{proof}We prove the results for $\CA$ and note that the same arguments hold true for $J_{X}$. Let $F$ be a finite subset of $\CA$. Since $\CA=\overline{\operatorname{span}}\{\chi_A: A \in \CB\}$, we may assume that $F=\{\chi_{A_1}, \cdots, \chi_{A_n}\}$ for some $n \in \N$. Since $\chi_{A_i} $ is projection and $\chi_{A_i}\chi_{A_j}=\chi_{A_i \cap A_j}$ for $1 \leq i,j \leq n$, we further assume that $F$ is closed under multiplication and taking adjoints. Then $C^*(F)=\overline{\operatorname{span}}(F)$ is finite dimensional. Thus, $\CA$ is a locally finite dimensional algebra. 

To prove that $K_1(\CA)=0$, we may assume that $\CA$ is unital by adding 1 to $\CB$ (if necessary).  Let $u$ be a unitary in $\CA$ and $0 < \epsilon < 1$. Then there exists a unitary $u'$ in $span_\Z\{\chi_A: A \in \CB\}$ such that $\|u -u'\|< \epsilon$. 
Since $\C$ is locally finite dimensional, one can choose a finite dimensional subalgebra $U$ of $\CA$ which contains $u'$ and the unit 1. 
Since the unitary group $\CU(U)$ of $U$ is connected, $u \sim_h 1$ in $\CU(U)$, and hence in $\CU(\CA)$. Since  $\|u -u'\|< \epsilon < 1$, we have  $u \sim_h u'$ in $\CU(\CA)$. Thus,  $u \sim_h 1$ in $\CU(\CA)$. The result then follows. 
\end{proof} 

The following is a generalization of \cite[Theorem 4.4]{BaCaPa2017}.

\begin{thm}
Let $(\CB, \CL, \theta, \CI_\af)$ be a generalized Boolean dynamical system, and define the linear map
\[
(1 - \Phi) : \operatorname{span}_\Z\{\chi_A : A \in \CB_{\mathrm{reg}}\} \longrightarrow \operatorname{span}_\Z\{\chi_A : A \in \CB\}
\]
by
\[
(1 - \Phi)(\chi_A) := \chi_A - \sum_{\af \in \Delta_A} \chi_{\theta_\af(A)} \quad \text{for all } A \in \CB_{\mathrm{reg}}.
\]
Then we have
 \begin{align*} K_1(C^*(\CB,\CL,\theta, \CI_\af)) &\cong \ker(1-\Phi), ~\text{and}~ \\
K_0(C^*(\CB,\CL,\theta, \CI_\af)) &\cong {\rm coker}(1-\Phi).
\end{align*}
The $K_0$-isomorphism is given by
\[
[p_A]_0 \mapsto \chi_A + \operatorname{Im}(1 - \Phi) \quad \text{for all } A \in \CB.
\]
\end{thm}

\begin{proof}By  \cite[Corollary 5.6(2)]{CaK2}, we have $C^*(\CB,\CL,\theta, \CI_\af) \cong \CO_{ X}$. We also have   $K_1(\CA)=K_1(J_{X})=0$ by Lemma \ref{K1-zero}. Then the result follows by 
 \cite[Theorem 8.6]{Ka2004c} and Lemma \ref{map Phi}. 
\end{proof}

\subsection{The K-theory of $C^*(\CB,\CL,\theta,\CI_\alpha;\CJ)$}  
Since $\phi(K_\CJ) \subset \mathfrak{K}(X)$, we have the following 6-term exact sequence by \cite[Proposition 8.7]{Ka2004c}:
$$\xymatrix{
K_0(K_\CJ) \ar[r]^{i_*-[X,K_\CJ]} & K_0 (\CA) \ar[r]^{\pi_*} & K_0(\CO_{(K_\CJ, X)}) \ar[d] & \\ 
 K_1(\CO_{(K_\CJ, X)}) \ar[u] & K_1 (\CA) \ar[l]^{\pi_*} & K_1(K_\CJ) \ar[l]^{i_*-[X,K_\CJ]},&  
}$$
where $\iota: K_\CJ \to \CA$ is the embedding, $\pi: \CA \to \CO_{(K_\CJ, X)}$ is the natural $*$-homomorphism, and $[X, K_\CJ]:K_*(K_\CJ) \to K_*(\CA)$ is defined by 
 $$[X, K_\CJ]:=(\iota_{\CA})_*^{-1} \circ (\iota_{\mathfrak{K}(X)})_* \circ (\phi|_{K_\CJ})_*$$
 (see \cite[Remark B.4]{Ka2004c}).
 As before we  identify  $K_0(K_\CJ)$ with $\operatorname{span}_\Z\{\chi_A: A \in \CJ\}$ and each $[\chi_A]_0 $ with $\chi_A$. Then we have the following.

\begin{prop}\label{K-theory of RGBDS} Let $(\CB,\CL,\theta, \CI_\af; \CJ)$ be a relative generalized Boolean dynamical system and identify $K_0(K_\CJ)$ with $\operatorname{span}_\Z\{\chi_A: A \in \CJ\}$. 
 Then the following hold:  
\begin{enumerate} 
\item  The homomorphism $[X, K_\CJ]: K_0(K_\CJ) \to K_0(\CA)$ induces the map $$\Phi|_\CJ: \operatorname{span}_\Z\{\chi_A: A \in \CJ\} \to \operatorname{span}_\Z\{\chi_A : A \in \CB\}$$ given by 
$$\chi_A \mapsto \sum_{\af \in \Delta_A} \chi_{\theta_\af(A)} ~\text{for}~A \in \CJ.$$
\item  Let $(1-\Phi)|_\CJ: \operatorname{span}_\Z\{\chi_A: A \in \CJ\} \to \operatorname{span}_\Z\{\chi_A: A \in \CB\}$ be the linear map given by  \[
    (1-\Phi)|_\CJ(\chi_A) = \chi_A - \sum_{\af \in \Delta_A} \chi_{\theta_\af(A)} \quad \text{for } A \in \CJ.
    \] Then, we have 
\begin{align*} K_1(C^*(\CB,\CL,\theta, \CI_\af;\CJ)) &\cong \ker((1-\Phi)|_\CJ) ~\text{and}~ \\
K_0(C^*(\CB,\CL,\theta, \CI_\af; \CJ)) &\cong {\rm coker}((1-\Phi)|_\CJ),
\end{align*}
 where the isomorphism on \(K_0\) is given by 
    \[
    [p_A]_0 \mapsto \chi_A + \operatorname{Im}((1-\Phi)|_\CJ), \quad A \in \CJ.
    \]
\end{enumerate}
\end{prop}

\begin{proof}(1) The same argument used in Lemma \ref{map Phi} gives that 
$(\iota_{\mathfrak{K}(X)})_* \circ (\phi|_{K_\CJ})_* ([\chi_A]_0)= \sum_{\af \in \Delta_A}  (\iota_{\CA})_*([\chi_{\theta_\af(A)}]_0)$ for $A \in \CJ$. 
Then our result follows from the given identifications and the definition of $[X,K_\CJ]$.

(2)  By the same arguments used in Lemma \ref{K1-zero}, we have $K_1(K_\CJ)=0$. Then we have the results from the above 6-term exact sequence. 
\end{proof}

\subsection{A description of the isomorphism class of $K_1$-group} 
Let $(\CB, \CL,\theta, \CI_\af; \CJ)$ be a relative generalized Boolean dynamical system  throughout this section. We   construct a  group isomorphism $\chi_1$ between 
$\ker((1-\Phi)|_\CJ)$ and $K_1(C^*(\CB,\CL,\theta,\CI_\af; \CJ))$.
To do that, for $x=\sum_{A \in F} k_A \chi_A \in \ker((1-\Phi)|_\CJ)$  where $F$ is a finite subset of $\CJ$,  we first define 
\begin{align*} L^+_x &:=\{ (A, \af, i) : A \in F, ~\theta_\af(A) \neq \emptyset,~ 1 \leq i \leq -k_A\} \cup \{(A,i) : A \in F, ~1 \leq i \leq k_A\}, \\
L^-_x &:=\{ (A, \af, i) : A \in F, ~\theta_\af(A) \neq \emptyset, ~1 \leq i \leq k_A\} \cup \{(A,i) : A \in F, ~1 \leq i \leq -k_A\}.
\end{align*}

\begin{prop}\label{same cardinality} When  $x\in \ker((1-\Phi)|_\CJ) $, then for any ultrafilter $\eta \in \widehat{\CB}$,  the sets 
$$L^+_{\eta}:=\{(A,\af,i) \in L^+_x : \theta_\af(A) \in \eta \} \cup \{(A,i) \in L^+_x : A \in \eta \}$$ and 
$$L^-_{\eta}:=\{(A,\af,i) \in L^-_x : \theta_\af(A) \in \eta \} \cup \{(A,i) \in L^-_x : A \in \eta \}$$ 
\noindent
are finite and have the same number of elements.
\end{prop}


\begin{proof}       If $x=\sum_{A\in F}k_A\chi_A\in span_\Z \{\chi_A:A\in \CJ\}$, then $x$ induces a continuous function $f_x: \widehat{\CB} \rightarrow \Z$ defined by 
 $f_x(\eta)=\sum_{A\in\eta}k_A$ for $\eta \in \widehat{\CB}$.   
      We then have that $f_{\Phi|_\CJ(x)}(\eta)=\sum_{A \in F}\sum_{\alpha \in \triangle_A: \theta_\alpha(A)\in\eta}k_A$. So if $x \in \ker((1-\Phi)|_\CJ)$, then $x=\Phi|_\CJ(x)$, and hence, 
\begin{eqnarray}\label{in-ker} f_x(\eta) =  \sum_{A\in \eta}k_A= \sum_{A\in F}\sum_{\alpha \in \triangle_A:\theta_\alpha(A)\in\eta}k_A =f_{\Phi|_\CJ(x)}(\eta).
\end{eqnarray} 
     \noindent 
  Put $x_+=\sum_{A\in F:k_A>0}k_A\chi_A$ and $x_-=\sum_{A\in F:k_A<0}k_A\chi_A$. Observe that
\begin{align*} |L^+_{\eta}| &= f_{x_+}(\eta)-f_{\Phi|_\CJ(x_-)}(\eta), \\
|L^-_{\eta}|&=f_{\Phi|_\CJ(x_+)}(\eta)-f_{x_-}(\eta)
\end{align*}
for any ultrafilter $\eta \in \widehat{\CB}$. Since  $f_x(\eta)=f_{x^+}(\eta)+f_{x_-}(\eta)$ and $f_{\Phi|_\CJ(x)}(\eta) =f_{\Phi|_\CJ(x_+)}(\eta)+f_{\Phi|_\CJ(x_-)}(\eta)$, we see that $|L^+_{\eta}|=|L^-_{\eta}|$
 from the equality (\ref{in-ker}).
\end{proof}

Let $\overline{F}:=F\cup\{\theta_\alpha(A):A\in F,\ \alpha\in\mathcal{L}\}$. Choose a finite family $\mathcal{F}\subseteq \CJ$ of mutually disjoint non-empty sets such that each $A\in \overline{F}$ is a disjoint union of sets from $\mathcal{F}$.
We then define  for each  $B \in\mathcal{F}$,  $$L_B^+:=L_\eta^+ ~\text{and}~ L_B^-:=L_\eta^-$$ where $\eta$ is an ultrafilter with $B\in\eta$. 
The sets $L_B^+$ and $L_B^-$ are well-defined for each $B \in\mathcal{F}$. In fact, since $B \in \CF$, we easily see  that for each $A \in \overline{F}$, 
\begin{eqnarray}\label{B in A} A \cap B \neq \emptyset ~\text{if and only if}~   B \subseteq A.\end{eqnarray}
 Thus if
$B\in\mathcal{F}$ and $\eta_1$ and $\eta_2$ are ultrafilters such that $B\in\eta_1$ and $B\in\eta_2$, we have 
$$ \eta_1 \cap \overline{F} = \eta_2 \cap \overline{F}.$$
It then follows that 
 $L_{\eta_1}^+=L_{\eta_2}^+$ and $L_{\eta_1}^-=L_{\eta_2}^-$.

 Denote 
by $h_B$ the common number of elements of $L_B^+$ and $L_B^-$. We then can define 
bijections 
$$\iota_B^+ : L_B^+ \rightarrow \{1,\dots,h_B\}\times\{B\} ~\text{and}~ \iota_B^- : L_B^- \rightarrow \{1,\dots,h_B\}\times\{B\}.$$ 
We also define the following sets  for Definition \ref{PV}:
\begin{align*} I_x^+ &:=\{(A,i,\alpha,B,C):A\in F,\ 1\le i\le k_A,\ \alpha\in\mathcal{L},\ B, C\in\mathcal{F},\ B\subseteq A,\ C\subseteq \theta_\alpha(A)\},\\
 I_x^-&:=\{(A,i,\alpha,B,C):A\in F,\ 1\le i\le -k_A,\ \alpha\in\mathcal{L},\ B, C\in\mathcal{F},\ B\subseteq A,\  C\subseteq \theta_\alpha(A)\}, \\
 J_x^+ &:=\{(A,i,B):A\in F,\ 1\le i\le k_A,\ B\in\mathcal{F},\ B\subseteq A\},\\
 J_x^-&:=\{(A,i,\alpha,C):A\in F,\ 1\le i\le -k_A,\ \alpha\in\mathcal{L},\ C\in\mathcal{F},\ C\subseteq \theta_\alpha(A)\}.
 \end{align*}

\begin{dfn}\label{PV} Suppose that $\mathfrak{A}$ is a $C^*$-algebra which contains a Toeplitz representation  $\{p_A:A\in\mathcal{B}\} \cup \{s_{\af,B}:\alpha\in\mathcal{L}, ~B \in \CI_\af\}$ of
$(\CB, \CL, \theta,\CI_\af)$. With the notation as above, we define two elements $V_x, P_x \in \mathsf{M}_{|N|}(\mathfrak{A})$ by
$$  V_x :=\sum_{I_x^+}p_Bs_{\alpha, \theta_\af(B)} p_C \mathsf{E}_{\iota_B^+(A,i),\iota_C^-(A,\alpha,i)}+\sum_{ I_x^-}p_Cs_{\alpha, \theta_\af(B)}^* p_B  \mathsf{E}_{\iota_C^+(A,\alpha,i),\iota_B^-(A,i)}$$
and
$$ P_x:=\sum_{ J_x^+}p_B  \mathsf{E}_{\iota_B^+(A,i),\iota_B^+(A,i)}+\sum_{J_x^-}p_C  \mathsf{E}_{\iota_C^+(A,\alpha,i),\iota_C^+(A,\alpha,i)}, $$
where $N:=\bigcup_{B\in\mathcal{F}}\{1,\dots,h_B\}\times\{B\}$  and $\{\mathsf{E}_{(i,B),(j,C)}\}_{B,C\in\mathcal{F},\ i\in\{1,\dots,h_B\},\ j\in\{1,\dots,h_C\}}$   are the standard matrix units in $\mathsf{M}_{|N|}(M(\mathfrak{A}))$ where $M(\mathfrak{A})$ is the multiplier algebra of $\mathfrak{A}$.
\end{dfn}

In what follows, we put 
\begin{align*} K_x^+ &:=\{(A,i,\alpha,C):A\in F,\ 1\le i\le k_A,\ \alpha\in\mathcal{L},\ C\in\mathcal{F},\ C\subseteq \theta_\alpha(A)\}, \\
K_x^-&:=\{(A,i,B):A\in F,\ 1\le i\le -k_A,\ B\in\mathcal{F},\ B\subseteq A\},\\
H_x^+ &:=\{(A,i,\alpha,B):A\in F,\ 1\le i\le k_A,\ \alpha\in\Delta_A,\ B\in\mathcal{F},\ B\subseteq A\},\\
H_x^-&:=\{(A,i,\alpha,B):A\in F,\ 1\le i\le -k_A,\ \alpha\in\Delta_A,\ B\in\mathcal{F},\ B\subseteq A\}.
\end{align*}

\begin{lem}\label{pv-prop} If $\{p_A, s_{\af,B}: A\in\mathcal{B}, ~\af \in \CL ~\text{and}~ B \in \CI_\af \}$  is a  Toeplitz representation of
$(\CB, \CL, \theta,\CI_\af)$,  then we have
\begin{align} \label{p} P_x = &\sum_{ K_x^+}p_C \mathsf{E}_{\iota_C^-(A,\alpha,i),\iota_C^-(A,\alpha,i)}+\sum_{ K_x^-}p_B \mathsf{E}_{\iota_B^-(A,i),\iota_B^-(A,i)},\\
\label{v*}V_x^*= &\sum_{ I_x^+} p_{C}s_{\af, \theta_\af(B)}^*p_B \mathsf{E}_{\iota_C^-(A,\alpha,i),\iota_B^+(A,i)}+\sum_{ I_x^-} p_Bs_{\af, \theta_\af(B)} p_C \mathsf{E}_{\iota_B^-(A,i),\iota_C^+(A,\alpha,i)},\\
\label{vv*}V_xV_x^* & = \sum_{ H_x^+} s_{\alpha, \theta_\af(B)} s_{\alpha, \theta_\af(B)}^*  \mathsf{E}_{\iota_B^+(A,i),\iota_B^+(A,i)} 
+  \sum_{ J_x^-}  p_C \mathsf{E}_{\iota_C^+(A,\alpha,i),\iota_{C}^+(A,\af,i)} ,\\
\label{v*v}V_x^*V_x & =\sum_{ K_x^+} p_C \mathsf{E}_{\iota_C^-(A,\alpha,i),\iota_C^-(A,\alpha,i)} + \sum_{H_x^-} s_{\alpha, \theta_\af(B)} s_{\alpha, \theta_\af(B)}^*\mathsf{E}_{\iota_B^-(A,i),\iota_B^-(A,i)}.
\end{align}
\end{lem}

\begin{proof} To show Equation (\ref{p}), fix $B \in \CF$. First note from (\ref{B in A}) that 
 \begin{align*} L_B^{\pm}  &= \{(A,i) \in  L^{\pm}_x: B\subseteq A\} \cup 
    \{(A,i,\alpha) \in L^{\pm}_x: B\subseteq \theta_\alpha(A)\},
   \end{align*}
where $\eta$ is an ultrafilter containing B.
Since  $|L_B^+|=|L_B^-|$ and $\iota_B^+$ and $\iota_B^-$ are bijections, we see that
\begin{align*} 
P_x &=\sum_{(A,i,B)\in J_x^+}p_B  \mathsf{E}_{\iota_B^+(A,i),\iota_B^+(A,i)}+\sum_{(A,i,\alpha,C)\in J_x^-}p_C  \mathsf{E}_{\iota_C^+(A,\alpha,i),\iota_C^+(A,\alpha,i)}\\
&= \sum_{(A,i)\in L_x^+, B \subseteq A} p_B\mathsf{E}_{\iota_B^+(A,i),\iota_B^+(A,i)} +
\sum_{(A,\af,i) \in L_x^+, C \subseteq \theta_\af(A)} p_C \mathsf{E}_{\iota_C^+(A,\af,i),\iota_C^+(A,\af,i)}\\
&= \sum_{(A',\af',j) \in L_x^-, B \subseteq \theta_{\af'}(A')} p_B \mathsf{E}_{\iota_B^-(A',\af',j),\iota_B^-(A',\af',j)} + \sum_{(A',j)\in L_x^-, C \subseteq A'} p_C\mathsf{E}_{\iota_C^-(A',j),\iota_C^-(A',j)}\\
&=\sum_{(A',j,\af',B) \in K_x^+}p_B \mathsf{E}_{\iota_B^-(A',\alpha',j),\iota_B^-(A',\alpha',j)}+\sum_{(A',j,C) \in K_x^-}p_C \mathsf{E}_{\iota_C^-(A',j),\iota_C^-(A',j)}.
\end{align*}

 Equation (\ref{v*}) is obvious. By using (\ref{v*}), the matrix unit relations, the fact that $\iota_B^{\pm}$ are bijections for each $B \in \CF$ and, the relations in Definition \ref{def:representation of RGBDS}, we have
\begin{align*} &V_xV_x^*\\
&= \Big( \sum_{ I_x^+}p_Bs_{\alpha, \theta_\af(B)} p_C \mathsf{E}_{\iota_B^+(A,i),\iota_C^-(A,\alpha,i)}+\sum_{ I_x^-}p_Cs_{\alpha, \theta_\af(B)}^* p_B  \mathsf{E}_{\iota_C^+(A,\alpha,i),\iota_B^-(A,i)} \Big)  \\
& \times \Big( \sum_{ I_x^+} p_{C'}s_{\alpha', \theta_{\alpha'}(B')}^*p_{B'} \mathsf{E}_{\iota_{C'}^-(A',\alpha',j),\iota_{B'}^+(A',j)}+\sum_{ I_x^-} p_{B'}s_{\alpha', \theta_{\alpha'}(B')} p_{C'} \mathsf{E}_{\iota_{B'}^-(A',j),\iota_{C'}^+(A',\alpha',j)} \Big)\\
&=\sum_{(A,i,\af,B,C),(A,i,\af,B',C) \in I_x^+} p_Bs_{\alpha, \theta_\af(B)} p_C s_{\alpha, \theta_\af(B')}^*p_{B'} \mathsf{E}_{\iota_B^+(A,i),\iota_{B'}^+(A,i)}\\
& \hskip5pc+ \sum_{(A,i,\alpha,B,C),(A,i,\af',B,C')\in I_x^-}  p_C s_{\alpha, \theta_\af(B)}^* p_B s_{\alpha', \theta_{\alpha'}(B)} p_{C'} \mathsf{E}_{\iota_C^+(A,\alpha,i),\iota_{C'}^+(A,\alpha',i)} \\
&=\sum_{(A,i,\alpha,B), (A,i,\af,B') \in H_x^+} p_B s_{\alpha, \theta_\af(B)} p_{\theta_\af(A) }  s_{\alpha, \theta_{\alpha}(B')}^*p_{B'} \mathsf{E}_{\iota_B^+(A,i),\iota_{B'}^+(A,i)} \\
& \hskip5pc + \sum_{(A,i,\alpha,B,C),(A,i,\af',B,C')\in I_x^-} p_C s_{\alpha, \theta_\af(B)}^* s_{\alpha', \theta_{\alpha'}(B)} p_{C'} \mathsf{E}_{\iota_C^+(A,\alpha,i),\iota_{C'}^+(A,\af',i)}\\ 
&=\sum_{(A,i,\alpha,B), (A,i,\af,B') \in H_x^+} p_B  s_{\alpha, \theta_\af(B)}  s_{\alpha, \theta_{\alpha}(B')}^*p_{B'} \mathsf{E}_{\iota_B^+(A,i),\iota_{B'}^+(A,i)} \\
& \hskip 5pc+\sum_{(A,i,\alpha,B,C),(A,i,\af,B,C')\in I_x^-} p_C p_{\theta_\af(B)} p_{C'} \mathsf{E}_{\iota_C^+(A,\alpha,i),\iota_{C'}^+(A,\af,i)} \\
&=\sum_{(A,i,\alpha,B), (A,i,\af,B') \in H_x^+} p_B s_{\alpha, \theta_\af(B)} p_{\theta_{\alpha}(B)}p_{\theta_{\alpha}(B')}  s_{\alpha, \theta_{\alpha}(B')}^*p_{B'} \mathsf{E}_{\iota_B^+(A,i),\iota_{B'}^+(A,i)} \\
& \hskip 5pc+ \sum_{(A,i,\alpha,B,C)\in I_x^-}  p_C p_{\theta_\af(B)} \mathsf{E}_{\iota_C^+(A,\alpha,i),\iota_{C}^+(A,\af,i)}\\
&=  \sum_{(A,i,\alpha,B)\in H_x^+} p_Bs_{\alpha, \theta_\af(B)} s_{\alpha, \theta_\af(B)}^* p_B \mathsf{E}_{\iota_B^+(A,i),\iota_B^+(A,i)} \\
& \hskip 5pc+  \sum_{(A,i,\alpha,C)\in J_x^-}  p_C p_{\theta_\af(A)} \mathsf{E}_{\iota_C^+(A,\alpha,i),\iota_{C}^+(A,\af,i)} \\
&=  \sum_{(A,i,\alpha,B)\in H_x^+} s_{\alpha, \theta_\af(B)} s_{\alpha, \theta_\af(B)}^*  \mathsf{E}_{\iota_B^+(A,i),\iota_B^+(A,i)} +
 \sum_{(A,i,\alpha,C)\in J_x^-}  p_C \mathsf{E}_{\iota_C^+(A,\alpha,i),\iota_{C}^+(A,\af,i)}.  
\end{align*}
The computation for $V_x^*V_x$ is similar.
\end{proof}

\begin{lem}\label{partial iso} If $\{p_A, s_{\af,B}: A\in\mathcal{B}, ~\af \in \CL ~\text{and}~ B \in \CI_\af \}$  is a  Toeplitz representation of
$(\CB, \CL, \theta,\CI_\af)$, then $V_x$ is a partial isometry with $P_xV_x=V_x=V_xP_x$.
\end{lem}

\begin{proof}By Equation (\ref{vv*}), we have 
\begin{align*} &(V_xV_x^*)V_x \\
&= \Big(\sum_{ H_x^+}p_Bs_{\alpha, \theta_\af(B)}s_{\alpha, \theta_\af(B)}^* p_B\mathsf{E}_{\iota_B^+(A,i),\iota_B^+(A,i)}+\sum_{J_x^-}p_C \mathsf{E}_{\iota_C^+(A,\alpha,i),\iota_C^+(A,\alpha,i)} \Big) \\
 &  \times \Big( \sum_{ I_x^+}p_{B'}s_{\alpha', \theta_{\af'}(B')} p_{C'} \mathsf{E}_{\iota_{B'}^+(A',j),\iota_{C'}^-(A',\alpha',j)} + \sum_{ I_x^-}p_{C'}s_{\alpha', \theta_{\af'}(B')}^* p_{B'}  \mathsf{E}_{\iota_{C'}^+(A',\alpha',j),\iota_{B'}^-(A',j)}\Big)\\
 &= \sum_{(A,i,\alpha,B,C'),(A,i,\af',B,C')\in I_x^+}p_Bs_{\alpha, \theta_\af(B)}s_{\alpha, \theta_\af(B)}^* p_B s_{\alpha', \theta_{\af'}(B)} p_{C'} \mathsf{E}_{\iota_B^+(A,i),\iota_{C'}^-(A,\alpha',i)} \\
 & \hskip9pc+ \sum_{ (A,i,\alpha,B',C)\in I_x^-}p_Cs_{\alpha, \theta_\af(B')}^* p_{B'}  \mathsf{E}_{\iota_C^+(A,\alpha,i),\iota_{B'}^-(A,i)} \\
&= \sum_{(A,i,\alpha,B,C'),(A,i,\af',B,C')\in I_x^+} p_Bs_{\alpha, \theta_\af(B)}s_{\alpha, \theta_\af(B)}^*  s_{\alpha', \theta_{\af'}(B)} p_{C'} \mathsf{E}_{\iota_B^+(A,i),\iota_{C'}^-(A,\alpha',i)} \\
 & \hskip 9pc+  \sum_{ (A,i,\alpha,B,C)\in I_x^-}p_Cs_{\alpha, \theta_\af(B)}^* p_{B}  \mathsf{E}_{\iota_C^+(A,\alpha,i),\iota_{B}^-(A,i)}\\
&= \sum_{(A,i,\alpha,B,C'),(A,i,\af,B,C')\in I_x^+} p_Bs_{\alpha, \theta_\af(B)}p_{\theta_\alpha(B)} p_{C'} \mathsf{E}_{\iota_B^+(A,i),\iota_{C'}^-(A,\alpha,i)} \\
 & \hskip 9pc+  \sum_{ (A,i,\alpha,B,C)\in I_x^-}p_Cs_{\alpha, \theta_\af(B)}^* p_{B}  \mathsf{E}_{\iota_C^+(A,\alpha,i),\iota_{B}^-(A,i)} \\
&= \sum_{(A,i,\alpha,B,C)\in I_x^+} p_Bs_{\alpha, \theta_\af(B)} p_{C} \mathsf{E}_{\iota_B^+(A,i),\iota_{C}^-(A,\alpha,i)} \\
 & \hskip 9pc+   \sum_{ (A,i,\alpha,B,C)\in I_x^-}p_Cs_{\alpha, \theta_\af(B)}^* p_{B}  \mathsf{E}_{\iota_C^+(A,\alpha,i),\iota_{B}^-(A,i)}\\
  &=V_x. 
\end{align*} This means that $V_x$ is a partial isomerty. It is straightforward to check that  $P_xV_x = V_x$ and $V_xP_x=V_x$.
\end{proof}

\begin{lem}\label{unitary}  Let $\{p_A, s_{\af,B}: A\in\mathcal{B}, ~\af \in \CL ~\text{and}~ B \in \CI_\af \}$ be the universal  $(\CB, \CL, \theta, \CI_\af;\CJ)$-representation. Then we have 
$V_xV_x^*=V_x^*V_x=P_x$, and hence $U_x:=V_x+(1-P_x)$ is  unitary.
\end{lem}

\begin{proof} It follows from Equation (\ref{vv*}) and Definition \ref{def:representation of RGBDS}(iv) that
 \begin{align*} V_xV_x^*
&=\sum_{H_x^+}s_{\alpha, \theta_\alpha(B)} s_{\alpha, \theta_\alpha(B)}^* \mathsf{E}_{\iota_B^+(A,i),\iota_B^+(A,i)}+\sum_{ J_x^-}p_C \mathsf{E}_{\iota_C^+(A,\alpha,i),\iota_C^+(A,\alpha,i)} \\
&=\sum_{ J_x^+} \Big( \sum_{\af \in \Delta_B} s_{\alpha, \theta_\alpha(B)} s_{\alpha, \theta_\alpha(B)}^* \Big) \mathsf{E}_{\iota_B^+(A,i),\iota_B^+(A,i)}+\sum_{ J_x^-}p_C \mathsf{E}_{\iota_C^+(A,\alpha,i),\iota_C^+(A,\alpha,i)} \\
&=\sum_{ J_x^+} p_B \mathsf{E}_{\iota_B^+(A,i),\iota_B^+(A,i)} +\sum_{ J_x^-}  p_C  \mathsf{E}_{\iota_C^+(A,\alpha,i),\iota_C^+(A,\alpha,i)} \\
&=P_x.
\end{align*}
Likewise, Equation (\ref{v*v}), Equation (\ref{p}) and  Definition \ref{def:representation of RGBDS}(iv) shows that 
\begin{align*}  V_x^*V_x 
&=\sum_{K_x^+}p_C \mathsf{E}_{\iota_C^-(A,\alpha,i),\iota_C^-(A,\alpha,i)}+\sum_{ H_x^-}s_{\alpha, \theta_\alpha(B)} s_{\alpha, \theta_\alpha(B)} ^* \mathsf{E}_{\iota_B^-(A,i),\iota_B^-(A,i)}, \\
&=\sum_{K_x^+}p_C \mathsf{E}_{\iota_C^-(A,\alpha,i),\iota_C^-(A,\alpha,i)}+\sum_{K_x^-} \Big( \sum_{\af \in \Delta_B}s_{\alpha, \theta_\alpha(B)} s_{\alpha, \theta_\alpha(B)}^* \Big) \mathsf{E}_{\iota_B^-(A,i),\iota_B^-(A,i)}, \\
&= \sum_{ K_x^+}p_C \mathsf{E}_{\iota_C^-(A,\alpha,i),\iota_C^-(A,\alpha,i)}+\sum_{K_x^-}p_B \mathsf{E}_{\iota_B^-(A,i),\iota_B^-(A,i)}\\
&=P_x.
\end{align*}Now, the fact that $U_x$ is unitary easily follows.
\end{proof}

 \begin{thm}\label{generator of K1} The map $\chi_1: \ker((1-\Phi)|_\CJ)\rightarrow K_1(C^*(\CB,\CL,\theta,\CI_\af; \CJ))$ given by 
$$\chi_1(x)=[U_x]_1$$
is a group isomorphism.
\end{thm}

\begin{proof}
Let $\CT:=\CT(\CB,\CL,\theta, \CI_\af)$, $\mathfrak{A}:=C^*(\CB,\CL,\theta, \CI_\af;\CJ)$ and  \(\pi: \mathcal{T} \to \mathfrak{A}\) denote the canonical quotient map, so that

$$\xymatrix{
0 \ar[r] & \ker\pi  \ar[r]^{i}  & \CT \ar[r]^{\pi} &  \mathfrak{A} \ar[r] & 0\\ 
}$$
is exact. Then the associated 6-term exact sequence is  
$$\xymatrix{
K_0(\ker \pi) \ar[r]^{i_*} & K_0 (\CT) \ar[r]^{\pi_*} & K_0( \mathfrak{A}) \ar[d] & \\ 
 K_1( \mathfrak{A}) \ar[u]^{\delta_1}& K_1 (\CT) \ar[l]^{\pi_*} & K_1(\ker \pi) \ar[l]^{i_*}.&  
}$$

Consider  the $C^*$-correspondences  $X (=X(\CB, \CL,\theta, \CI_\alpha))$ over $\CA (=\CA(\CB, \CL,\theta))$.
 Then $\CT_{X}$ is isomorphic to $\CT$ and $\CO_{(K_\CJ,X)}$ is isomorphic to $\mathfrak{A}$. It therefore follows from the results of \cite[Section 8]{Ka2004c} that $\ker\pi$ and $\CT$ are $\it{KK}$-equivalent to  $ K_\CJ$ and $\CA$, respectively, and that there are group isomorphisms $\kappa: K_0(\ker \pi) \rightarrow \operatorname{span}_{\mathbb{Z}}\{\chi_A: A \in \CJ\}$ and $\ld: K_0(\CT) \rightarrow \operatorname{span}_{\mathbb{Z}}\{\chi_A: A \in \CB\} $ such that the diagram

$$\xymatrix  @C=18pt {0 \ar[r] & K_1( \mathfrak{A}) \ar[r]^{\delta_1} &  K_0(\ker \pi) \ar[d]^{\kappa} \ar[r]^{i_*}& K_0 (\CT) \ar[d]^{\ld} \ar[r]^{\pi_*} & K_0( \mathfrak{A})   \ar[r] &  0\\
& & \operatorname{span}_{\mathbb{Z}}\{\chi_A: A \in \CJ\} \ar[r]_{(1-\Phi)|_\CJ} &   \operatorname{span}_{\mathbb{Z}}\{\chi_A: A \in \CB\}  & &
}$$
commutes with the top row exact, and such that $\ld([p_A]_0)=\chi_A$ for $A\in \CB$, and 
\begin{align}\label{kappa map}
\kappa\Big(\Big[ p_B-\sum_{\af \in \Delta_B} s_{\af, \theta_\af(B)} s_{\af, \theta_\af(B)} ^* \Big]_0 \Big) = \chi_B
\end{align}
for $B \in \CJ$.

To prove that $\chi_1$ is a group isomorphism, we fix $x \in \ker((1-\Phi)|_\CJ)$ and lift $U_x=V_x+(1-P_x) \in \mathsf{M}_{|N|}(\mathfrak{A})$ to $\widetilde{U_x}=\widetilde{V_x}+(1-\widetilde{P_x}) \in \mathsf{M}_{|N|}(\CT)$ where $\widetilde{V_x}$ and $\widetilde{P_x}$ are the elements  in $\mathsf{M}_{|N|}(\CT)$ we get in Definition \ref{PV} by using the universal Toeplitz  representation $\{p_A,s_{\af,B}: A \in \CB, \af \in \CL, B \in \CI_\af\}$. By Lemma \ref{partial iso}, $\widetilde{V_x}$ is partial isometry with $\widetilde{P_x}\widetilde{V_x}=\widetilde{V_x}\widetilde{P_x}=\widetilde{V_x}$. It follows that $\widetilde{U_x}$ is also a partial isometry. We then have by Lemma \ref{pv-prop} that 
\begin{align*} 1-\widetilde{U_x}\widetilde{U_x^*}&=\widetilde{P_x}-\widetilde{V_x}\widetilde{V_x^*} \\
&= \sum_{(A,i,B)\in J_x^+}p_B  \mathsf{E}_{\iota_B^+(A,i),\iota_B^+(A,i)}-\sum_{(A,i,\af,B) \in H_x^+}s_{\alpha, \theta_\alpha(B)}s_{\alpha, \theta_\alpha(B)}^* \mathsf{E}_{\iota_B^+(A,i),\iota_B^+(A,i)} \\
&= \sum_{(A,i,B)\in J_x^+} \Big( p_B-\sum_{\af \in \Delta_B}s_{\alpha, \theta_\alpha(B)}s_{\alpha, \theta_\alpha(B)}^* \Big)\mathsf{E}_{\iota_B^+(A,i),\iota_B^+(A,i)}
\end{align*}
and that 
\begin{align*} 1-\widetilde{U_x^*}\widetilde{U_x}&=\widetilde{P_x}-\widetilde{V_x^*}\widetilde{V_x} \\
&=\sum_{(A,i,B) \in K_x^-}p_B \mathsf{E}_{\iota_B^-(A,i),\iota_B^-(A,i)}-\sum_{(A,i,\af,B) \in H_x^-} s_{\alpha, \theta_\alpha(B)}s_{\alpha, \theta_\alpha(B)}^*\mathsf{E}_{\iota_B^-(A,i),\iota_B^-(A,i)}\\
&=   \sum_{(A,i,B)\in K_x^-} \Big( p_B-\sum_{\af \in \Delta_B}s_{\alpha, \theta_\alpha(B)}s_{\alpha, \theta_\alpha(B)}^* \Big)\mathsf{E}_{\iota_B^-(A,i),\iota_B^-(A,i)}.
\end{align*}
Thus, in $K_0(\ker\pi)$ we have that 
\begin{eqnarray}\label{index map}[1-\widetilde{U_x}\widetilde{U_x^*}]_0-[1-\widetilde{U_x^*}\widetilde{U_x}]_0 = \sum_{k_A \neq 0, B \subseteq A }k_A\Big[  p_B-\sum_{\af \in \Delta_B} s_{\alpha, \theta_\alpha(B)}s_{\alpha, \theta_\alpha(B)}^* \Big]_0.
\end{eqnarray}
This together with Equation (\ref{kappa map}) and (\ref{index map}) implies that 
\begin{eqnarray}\label{left inverse}\kappa \circ \delta_1 \circ \chi_1(x)=x 
\end{eqnarray}
for any $x \in \ker((1-\Phi)|_\CJ)$. This shows that $\chi_1$ is injective. 

Now, fix $\bf y$ $ \in K_1(\mathfrak{A})$ and observe that  
$$((1-\Phi)|_\CJ) \circ \kappa \circ \delta_1(\bf y)=\ld \circ \iota_* \circ \delta_1 (\bf y) = 0$$  so that 
$\bf z:=\kappa \circ \delta_1(\bf y)  $ lies in $\ker((1-\Phi)|_\CJ)$. 
We then see from Equation (\ref{left inverse}) that 
$$ \kappa \circ \delta_1 \circ \chi_1 ( \bf z)=\bf z =(\kappa \circ \delta_1)(\bf y),$$
and hence $\chi_1 (\bf z)= \bf y$ since  $\kappa \circ \delta_1$ is injective. 
Thus we conclude that  $\kappa \circ \delta_1$ is an inverse of $\chi_1$, which shows that $\chi_1$ is a group isomorphism. 
\end{proof}

\section{Representation of Gauge-Invariant Ideals in Relative Generalized Boolean Dynamical Systems}\label{Morita}

Let $(\CB,\CL,\theta,\CI_\alpha;\CJ)$ be a relative generalized Boolean dynamical system. We will construct an extension $(\tilde{\CB},\CL,\tilde{\theta},\tilde{\CI}_\alpha;\tilde{\CJ})$ of $(\CB,\CL,\theta,\CI_\alpha;\CJ)$ such that $C^*(\tilde{\CB},\CL,\tilde{\theta},\tilde{\CI}_\alpha;\tilde{\CJ})$ is isomorphic to  $C^*(\CB,\CL,\theta,\CI_\alpha;\CJ)$ and such that there for each gauge-invariant ideal $I$ in $C^*(\CB,\CL,\theta,\CI_\alpha;\CJ)$ is a hereditary $\tilde{\CJ}$-saturated ideal $\tilde{\CH}_I$ of $\tilde{\CB}$ such that $C^*(\tilde{\CH}_I,\CL,\tilde{\theta},\tilde{\CH}_I\cap\tilde{\CI}_\alpha;\tilde{\CH}_I\cap\tilde{\CJ})$ is isomorphic to a full hereditary $C^*$-subalgebra of $I$.

Let $\prod_{\beta\in\CL^*}\CR_\beta$ be the Boolean algebra consisting of families $(A_\beta)_{\beta\in\CL^*}$ where $A_\beta\in\CR_\beta$ for each $\beta\in\CL^*$, and union, intersection and relative complement defined by 
\begin{align*}
(A_\beta)_{\beta\in\CL^*}\cup (B_\beta)_{\beta\in\CL^*}&=(A_\beta\cup B_\beta)_{\beta\in\CL^*},\\
(A_\beta)_{\beta\in\CL^*}\cap (B_\beta)_{\beta\in\CL^*}&=(A_\beta\cap B_\beta)_{\beta\in\CL^*},\\
(A_\beta)_{\beta\in\CL^*}\setminus (B_\beta)_{\beta\in\CL^*}&=(A_\beta\setminus B_\beta)_{\beta\in\CL^*}.
\end{align*}
Define for each $\alpha\in\CL$ a Boolean homomorphism $\overline{\theta}_{\alpha}:\prod_{\beta\in\CL^*}\CR_\beta\to\prod_{\beta\in\CL^*}\CR_\beta$ by 
\begin{equation}
\overline{\theta}_{\alpha}\left((A_\beta)_{\beta\in\CL^*}\right)=(A_{\alpha\beta})_{\beta\in\CL^*}.
\end{equation}


Let $\overline{\CB}$ be the Boolean subalgebra of $\prod_{\beta\in\CL^*}\CR_\beta$ consisting of those $(A_\beta)_{\beta\in\CL^*}\in \prod_{\beta\in\CL^*}\CR_\beta$ for which $A_{\beta\alpha}=\theta_{\alpha}(A_\beta)$ for all but finitely many $\alpha\in\CL$ and finitely many $\beta\in\CL^*$.

Let $\alpha\in\CL$. Since $\overline{\theta}_\alpha(\overline{\CB})\subseteq \overline{\CB}$, the Boolean homomorphism $\overline{\theta}_\alpha$ restricts to a Boolean homomorphism from $\overline{\CB}$ to $\overline{\CB}$. We will also use $\overline{\theta}_\alpha$ to denote the restriction of $\overline{\theta}_\alpha$ to $\overline{\CB}$. 

Let $\overline{\CI}$ be the subset of $\overline{\CB}$ consisting of those elements $(A_\beta)_{\beta\in\CL^*}\in\overline{\CB}$ for which $A_\beta\in\CJ$ for all $\beta\in\CL^*$, and $A_\beta=\emptyset$ for all but finitely many $\beta\in\CL^*$. 
Then $\overline{\CI}$ is an ideal of $\overline{\CB}$. Let $\tilde{\CB}$ be the quotient Boolean algebra $\overline{\CB}/\overline{\CI}$ and let $q$ be the quotient map from $\overline{\CB}$ to $\tilde{\CB}$. 
Since $\overline{\theta}_\alpha(\overline{\CI})\subseteq\overline{\CI}$ for each $\alpha\in\CL$, $\overline{\theta}$ descents to an action of $\CL$ on $\tilde{\CB}$. We denote this action by $\tilde{\theta}$. Then $(\tilde{\CB},\CL,\tilde{\theta})$ is a Boolean dynamical system, and 
\begin{equation}\tilde{\CR}_\alpha:=\CR^{(\tilde{\CB},\CL,\tilde{\theta})}_\alpha=q(\{(A_\beta)_{\beta\in\CL^*}\in\overline{\CB}:A_\beta\in\CR_{\alpha\beta}\text{ for each }\beta\in\CL^*\})
\end{equation} for $\alpha\in\CL$.

For $\alpha\in\CL$, we let $\overline{\CI}_\alpha:=\{(A_\beta)_{\beta\in\CL^*}\in\overline{\CB}:A_\beta\in\CI_{\alpha\beta}\text{ for }\beta\in\CL^*\}$ and $\tilde{\CI}_\alpha:=q(\overline{\CI}_\alpha)$. Moreover, we let $\overline{\CJ}:=\{(A_\beta)_{\beta\in\CL^*}\in\overline{\CB}:A_\emptyset\in\CJ\}$ and $\tilde{\CJ}=q(\overline{\CJ})$. Then it is clear that $(\tilde{\CB},\CL,\tilde{\theta},\tilde{\CI}_\alpha;\tilde{\CJ})$ is a relative generalized Boolean dynamical system.

For $\gamma\in\CL^*$, define $\iota_\gamma:\CR_\gamma\to\tilde{\CB}$ by $\iota_\gamma(A):=q((B_\beta)_{\beta\in\CL^*})\in\tilde{\CB}$ where 
$$B_\bt=
\left\{
   \begin{array}{ll}
      \theta_{\beta'}(A) & \hbox{if\ }  \beta=\gamma\beta' ~\text{for some}~ \beta'\in\CL^*,\\
     \emptyset  & \hbox{if\ }  \beta ~\text{ does not begin with}~ \gamma.
       \end{array}
\right.
$$  
 Then $\iota_\gamma$ is a Boolean homomorphism. We let $\CR_\emptyset:=\CB$ and $\iota:=\iota_\emptyset$. Then $\iota\circ\theta_\alpha=\tilde{\theta}_\alpha\circ\iota$ for all $\alpha\in\CL$. In fact, if $A\in\CB$, then $\tilde{\theta}_\alpha(\iota(A))=q((\theta_{\alpha\beta}(A))_{\beta\in\CL^*})$ and $\iota(\theta_\alpha(A))=q(\theta_\beta(\theta_\alpha(A))_{\beta\in\CL^*})=q((\theta_{\alpha\beta}(A))_{\beta\in\CL^*})$.

\begin{thm}\label{prop:one}
Let $(\CB,\CL,\theta,\CI_\alpha;\CJ)$ be a relative generalized Boolean dynamical system and let $(\tilde{\CB},\CL,\tilde{\theta},\tilde{\CI}_\alpha;\tilde{\CJ})$ and $\iota$ be as above. Then $\iota$ is injective, and there is $*$-isomorphism from $\phi:C^*(\CB,\CL,\theta,\CI_\alpha;\CJ)\to C^*(\tilde{\CB},\CL,\tilde{\theta},\tilde{\CI}_\alpha;\tilde{\CJ})$ such that
\begin{align*}  \phi(p_A)&=p_{\iota(A)}  ~\text{for}~  A\in\CB , \\
\phi(s_{\alpha,B})&=s_{\alpha,\iota(B)} ~\text{for}~ \alpha\in\CL ~\text{and}~ B\in\CI_\alpha, \\
\phi(s_{\gamma,C}s_{\gamma,C}^*)&=p_{\iota_\gamma(C)} ~\text{for}~ \gamma\in\CL^* ~\text{and}~ C\in\CR_\gamma. 
\end{align*}
\end{thm}

\begin{proof}
We first prove that $\iota$ is injective. Since $\iota$ is a Boolean homomorphism, it suffices to prove that $\iota(A)\ne\emptyset$ for $A\ne\emptyset$. Assume for contradiction that $A\ne\emptyset$ and $\iota(A)\in\overline{\CI}$. Then $A\in\CJ$. It follows that there is an $\alpha_1\in\CL$ such that $\emptyset\ne\theta_{\alpha_1}(A)\in\CJ$. Continuing like this, we get an infinite sequence $\alpha_1,\alpha_2,\dots$ of elements in $\CL$ such that $\theta_{\alpha_1\alpha_2\dots\alpha_n}(A)\ne\emptyset$ for every $n$. But this is a contradiction to the assumption that $\iota(A)\in\overline{\CI}$. We thus have that $\iota$ is injective.

Next, we construct the $*$-isomorphism $\phi$. For $A\in\CB$, $\alpha\in\CL$, and $B\in\CI_\alpha$, let $$P_A=p_{\iota(A)} ~\text{and}~S_{\alpha,B}=s_{\alpha,\iota(B)}.$$ Using that $\iota:\CB\to\tilde{\CB}$ is an injective Boolean homomorphism such that $\iota(\CJ)\subseteq\tilde{\CJ}$, and $\iota\circ\theta_\alpha=\tilde{\theta}_\alpha\circ\iota$ for all $\alpha\in\CL$, it is easy to check that $\{P_A,\ S_{\alpha,B}:A\in\CB,\ \alpha\in\CL,\ B\in\CI_\alpha\}$ is a $(\CB,\CL,\theta,\CI_\alpha;\CJ)$-representation.
 It thus follows from the universal property of $C^*(\CB,\CL,\theta,\CI_\alpha;\CJ)$ that there is a $*$-homomorphism $\phi:C^*(\CB,\CL,\theta,\CI_\alpha;\CJ)\to C^*(\tilde{\CB},\CL,\tilde{\theta},\tilde{\CI}_\alpha;\tilde{\CJ})$ such that $$\phi(p_A)=p_{\iota(A)} ~\text{and}~ \phi(s_{\alpha,B})=s_{\alpha,\iota(B)}$$
  for $A\in\CB$, $\alpha\in\CL$ and $B\in\CI_\alpha$.

Since $\iota$ is injective and $\iota(A)\in\tilde{\CJ}\iff A\in\CJ$ for $A\in\CB$, it follows from the Gauge-Invariant Uniqueness Theorem that $\phi$ is injective. 

We now show that $\phi$ is surjective. Since $C^*(\tilde{\CB},\CL,\tilde{\theta},\tilde{\CI}_\alpha;\tilde{\CJ})$ is generated by $$\left\{p_{(A_\beta)_{\beta\in\CL^*}},\ s_{\alpha,(B_\beta)_{\beta\in\CL^*}}:(A_\beta)_{\beta\in\CL^*}\in\tilde{\CB},\ \alpha\in\CL,\ (B_\beta)_{\beta\in\CL^*}\in\tilde{\CI}_\alpha\right\},$$ it suffices to show that $p_{(A_\beta)_{\beta\in\CL^*}}$ and $s_{\alpha,(B_\beta)_{\beta\in\CL^*}}$ belong to the image of $\phi$ for $(A_\beta)_{\beta\in\CL^*}\in\tilde{\CB}$, $\alpha\in\CL$, and $(B_\beta)_{\beta\in\CL^*}\in\tilde{\CI}_\alpha$.

Suppose that $\gamma=\gamma_1\gamma_2\cdots\gamma_n\in\CL^*\setminus\{\emptyset\}$ and $A\in\CR_\gamma$. Then 
\begin{align*}
\iota_\gamma(A)&\in\tilde{\CJ},\ \Delta_{\iota_\gamma(A)}=\{\gamma_1\},\ \tilde{\theta}_{\gamma_1}(\iota_\gamma(A))=\iota_{\gamma_2\gamma_3\cdots\gamma_n}(A),\ \iota_{\gamma_2\gamma_3\cdots\gamma_n}(A)\in\tilde{\CJ},\\ 
&\Delta_{\iota_{\gamma_2\gamma_3\cdots\gamma_n}(A)}=\{ \gamma_2\},\ \tilde{\theta}_{\gamma_2}(\iota_{\gamma_2\gamma_3\cdots\gamma_n}(A))=\iota_{\gamma_3\gamma_4\cdots\gamma_n}(A),\ \dots,\ \iota_{\gamma_n}(A)\in\tilde{\CJ},\\ 
&\Delta_{\iota_{\gamma_n}(A)}=\{\gamma_n\}, \text{ and }\tilde{\theta}_{\gamma_n}(\iota_{\gamma_n}(A))=\iota(A).
\end{align*} 
Thus, $n$ applications of the Cuntz--Krieger relations yells 
\begin{equation}\label{eq:1}
\begin{split}
p_{\iota_\gamma(A)}&=s_{\gamma_1,\iota_{\gamma_2\gamma_3\cdots\gamma_n}(A)}s_{\gamma_1,\iota_{\gamma_2\gamma_3\cdots\gamma_n}(A)}^*\\
&=s_{\gamma_1,\iota_{\gamma_2\gamma_3\cdots\gamma_n}(A)}p_{\iota_{\gamma_2\gamma_3\cdots\gamma_n}(A)}s_{\gamma_1,\iota_{\gamma_2\gamma_3\cdots\gamma_n}(A)}^*\\
&=s_{\gamma_1,\iota_{\gamma_2\gamma_3\cdots\gamma_n}(A)}s_{\gamma_2,\iota_{\gamma_3\gamma_4\cdots\gamma_n}(A)}s_{\gamma_2,\iota_{\gamma_3\gamma_4\cdots\gamma_n}(A)}^*s_{\gamma_1,\iota_{\gamma_2\gamma_3\cdots\gamma_n}(A)}^*\\
&\hspace{12em}\vdots\\
&=s_{\gamma_1,\iota_{\gamma_2\gamma_3\cdots\gamma_n}(A)}s_{\gamma_2,\iota_{\gamma_3\gamma_4\cdots\gamma_n}(A)}\cdots s_{\gamma_n,\iota(A)}s_{\gamma_n,\iota(A)}^*\cdots\\
&\hspace{4em} s_{\gamma_2,\iota_{\gamma_3\gamma_4\cdots\gamma_n}(A)}^*s_{\gamma_1,\iota_{\gamma_2\gamma_3\cdots\gamma_n}(A)}^*\\
&=s_{\gamma,\iota(A)}s_{\gamma,\iota(A)}^*.
\end{split}
\end{equation}

Suppose then that $\alpha\in\CL$, $\gamma\in\CL^*$, and $B\in\CI_{\alpha\gamma}$. Then there is a $C\in\CI_\alpha$ such that $B\subseteq\theta_\gamma(C)$. Then $\iota_\gamma(B)\subseteq\iota(C)$ and
\begin{equation}\label{eq:2}
s_{\alpha,\iota_\gamma(B)}
=s_{\alpha,\iota(C)}p_{\iota_\gamma(B)}.
\end{equation}
By using \eqref{eq:1} and \eqref{eq:2}, it is straightforward to use induction to prove that
 \begin{align}\label{induction}p_{\iota_\gamma(A)}~\text{and}~ s_{\alpha,\iota_\gamma(B)} ~\text{ belong to the image of}~\phi 
 \end{align}
  for all $\gamma\in\CL^*$, $A\in\CR_\gamma$, $\alpha\in\CL$, and $B\in\CI_{\alpha\gamma}$;  since $\phi(p_A)=p_{\iota(A)}$ and $\phi(s_{\alpha,B})=s_{\alpha,\iota(B)}$, (\ref{induction}) clearly hols for $\gm=\emptyset$, $A\in \CB$, $\alpha\in\CL$, and $B\in\CI_{\alpha}$. 
Suppose (\ref{induction}) hols for $\gamma\in\CL^n$, $A\in\CR_\gamma$, $\alpha\in\CL$, and $B\in\CI_{\alpha\gamma}$. Then for $\gm \in \CL^{n+1}$ and  $A\in\CR_\gamma$,
\begin{align*}
p_{\iota_\gamma(A)}&=s_{\gamma_1,\iota_{\gamma_2\gamma_3\cdots\gamma_{n+1}}(A)}s_{\gamma_1,\iota_{\gamma_2\gamma_3\cdots\gamma_{n+1}}(A)}^* \in \operatorname{Im}(\phi),
\end{align*}
and hence, \begin{align*}s_{\alpha,\iota_\gamma(B)}
=s_{\alpha,\iota(C)}p_{\iota_\gamma(B)}   \in \operatorname{Im}(\phi)\end{align*}
for any $\alpha\in\CL$, and $B\in\CI_{\alpha\gamma}$. 

 Since $\tilde{\CB}$ is generated by $\bigcup_{\gamma\in\CL^*}\iota_\gamma(\CR_\gamma)$, and $\tilde{\CI}_\alpha$ is, for each $\alpha\in\CL$, generated by $\bigcup_{\gamma\in\CL^*}\iota_\gamma(\CI_{\alpha\gamma})$, we conclude that $p_{(A_\beta)_{\beta\in\CL^*}}$ and $s_{\alpha,(B_\beta)_{\beta\in\CL^*}}$ belong to the image of $\phi$ for all $(A_\beta)_{\beta\in\CL^*}\in\tilde{\CB}$, $\alpha\in\CL$, and $(B_\beta)_{\beta\in\CL^*}\in\tilde{\CI}_\alpha$, and thus that $\phi$ is surjective.

Finally, since $\phi(s_{\alpha,B})=s_{\alpha,\iota(B)}$ for $\alpha\in\CL$ and $B\in\CI_\alpha$, we see that $\phi(s_{\gamma,B})=s_{\gamma,\iota(B)}$ for $\gamma\in\CL^*$ and $B\in\CI_\gamma$. It thus follows from \eqref{eq:1} that $\phi(s_{\gamma,C}s_{\gamma,C}^*)=p_{\iota_\gamma(C)}$ for $\gamma\in\CL^*$ and $C\in\CR_\gamma$.
\end{proof}

Suppose now that $\CH$ is a hereditary $\CJ$-saturated ideal of $\CB$ and $\CS$ is an ideal of $\CB_\CH$ such that $\CH\cup\CJ\subseteq\CS$. For $A\in\CS$, we let $\Delta_{[A]}:=\Delta_{[A]_\CH}^{(\CB/\CH,\CL,\theta)}$. Let $\tilde{\CS}$ be the Boolean subalgebra of $\tilde{\CB}$ generated by 
\begin{equation*}
\bigl\{\iota(A)\setminus\cup_{\alpha\in\Delta_{[A]}}\iota_\alpha(\theta_\alpha(A)):A\in\CS\bigr\}.
\end{equation*}
Suppose $A\in\CS$ and $\alpha'\in\CL$. If $\alpha'\in\Delta_{[A]}$, then 
\begin{equation*}
\tilde{\theta}_{\alpha'}\bigl(\iota(A)\setminus\cup_{\alpha\in\Delta_{[A]}}\iota_\alpha(\theta_\alpha(A))\bigr)=\emptyset.
\end{equation*}
If $\alpha'\notin\Delta_{[A]}$, then $\theta_{\alpha'}(A)\in\CH$, $\Delta_{[\theta_{\alpha'}(A)]}=\emptyset$, and
\begin{align*}
\tilde{\theta}_{\alpha'}\bigl(\iota(A)\setminus\cup_{\alpha\in\Delta_{[A]}}\iota_\alpha(\theta_\alpha(A))\bigr)&=\iota(\theta_{\alpha'}(A))\\
&=\iota(\theta_{\alpha'}(A))\setminus\cup_{\alpha\in\Delta_{[\theta_{\alpha'}(A)]}}\iota_\alpha(\theta_\alpha(\theta_{\alpha'}(A)))\in\tilde{\CS}.
\end{align*}
We thus have that $\theta_{\alpha'}(\tilde{\CS})\subseteq\tilde{\CS}$. So $\tilde{\theta}$ restricts to an action of $\CL$ on $\tilde{\CS}$. We will just write $\tilde{\theta}$ for the restriction of $\tilde{\theta}$ to $\tilde{\CS}$. We then have that $(\tilde{\CS},\CL,\tilde{\theta},\tilde{\CI}_\alpha\cap\tilde{\CS};\tilde{\CJ}\cap\tilde{\CS})$ is a relative generalized Boolean dynamical system.

\begin{thm}\label{gii}
Let $(\CB,\CL,\theta,\CI_\alpha;\CJ)$ be a relative generalized Boolean dynamical system, let $\CH$ be a hereditary $\CJ$-saturated ideal of $\CB$, and let $\CS$ be an ideal of $\CB_\CH$ such that $\CH\cup\CJ\subseteq\CS$. Let $\tilde{\CS}$ be as above. 

Then there is an injective $*$-homomorphism $\psi:C^*(\tilde{\CS},\CL,\tilde{\theta},\tilde{\CI}_\alpha\cap\tilde{\CS};\tilde{\CJ}\cap\tilde{\CS})\to C^*(\CB,\CL,\theta,\CI_\alpha;\CJ)$ such that 
$$\psi\bigl(p_{\iota(A)\setminus\cup_{\alpha\in\Delta_{[A]}}\iota_\alpha(\theta_\alpha(A))}\bigr)=p_A-\sum_{\alpha\in\Delta_{[A]}}s_{\alpha,\theta_\alpha(A)}s_{\alpha,\theta_\alpha(A)}^*$$
for $A\in\CS$, and
$$\psi\bigl(s_{\alpha',\iota(B)\setminus\cup_{\alpha\in\Delta_{[B]}}\iota_\alpha(\theta_\alpha(B))}\bigr)=s_{\alpha',B}-s_{\alpha,'B}\sum_{\alpha\in\Delta_{[B]}}s_{\alpha,\theta_\alpha(B)}s_{\alpha,\theta_\alpha(B)}^*$$
for $\alpha'\in\CL$ and $B\in\CI_{\alpha'}\cap\CS$.

The image of $\psi$ is a full hereditary $C^*$-subalgebra of $I_{(\CH,\CS)}$, and $I_{(\CH,\CS)}$ and $C^*(\tilde{\CS},\CL,\tilde{\theta},\tilde{\CI}_\alpha\cap\tilde{\CS};\tilde{\CJ}\cap\tilde{\CS})$ are Morita equivalent.
\end{thm}

\begin{proof}
Let $\eta:C^*(\tilde{\CB},\CL,\tilde{\theta},\tilde{\CI}_\alpha;\tilde{\CJ})\to C^*(\CB,\CL,\theta,\CI_\alpha;\CJ)$ be the inverse of the $*$-isomorphism $\phi:C^*(\CB,\CL,\theta,\CI_\alpha;\CJ)\to C^*(\tilde{\CB},\CL,\tilde{\theta},\tilde{\CI}_\alpha;\tilde{\CJ})$ constructed in Theorem~\ref{prop:one}. For $\tilde{A}\in\tilde{\CS}$, $\alpha\in\CL$, and $\tilde{B}\in\tilde{\CI}_\alpha\cap\tilde{\CS}$, let $P_{\tilde{A}}:=\eta(p_{\tilde{A}})$ and $S_{\alpha, \tilde{B}}:=\eta(s_{\alpha, \tilde{B}})$. Then $\{P_{\tilde{A}},\ S_{\alpha, \tilde{B}}:\tilde{A}\in\tilde{\CS},\ \alpha\in\CL,\ \tilde{B}\in\tilde{\CI}_\alpha\cap\tilde{\CS}\}$ is a $(\tilde{\CS},\CL,\tilde{\theta},\tilde{\CI}_\alpha\cap\tilde{\CS};\tilde{\CJ}\cap\tilde{\CS})$-representation. It thus follows from the universal property of $C^*(\tilde{\CS},\CL,\tilde{\theta},\tilde{\CI}_\alpha\cap\tilde{\CS};\tilde{\CJ}\cap\tilde{\CS})$ that there is $*$-homomorphism $\psi:C^*(\tilde{\CS},\CL,\tilde{\theta},\tilde{\CI}_\alpha\cap\tilde{\CS};\tilde{\CJ}\cap\tilde{\CS})\to C^*(\CB,\CL,\theta,\CI_\alpha;\CJ)$ such that 
\begin{align*}
\psi\bigl(p_{\iota(A)\setminus\cup_{\alpha\in\Delta_{[A]}}\iota_\alpha(\theta_\alpha(A))}\bigr)&=P_{\iota(A)\setminus\cup_{\alpha\in\Delta_{[A]}}\iota_\alpha(\theta_\alpha(A))}=\eta\bigl(p_{\iota(A)\setminus\cup_{\alpha\in\Delta_{[A]}}\iota_\alpha(\theta_\alpha(A))}\bigr)\\
&=p_A-\sum_{\alpha\in\Delta_{[A]}}s_{\alpha,\theta_\alpha(A)}s_{\alpha,\theta_\alpha(A)}^*
\end{align*}
for $A\in\CS$, and
\begin{align*}
\psi\bigl(s_{\alpha',\iota(B)\setminus\cup_{\alpha\in\Delta_{[B]}}\iota_\alpha(\theta_\alpha(B))}\bigr)&=S_{\alpha',\iota(B)\setminus\cup_{\alpha\in\Delta_{[B]}}\iota_\alpha(\theta_\alpha(B))}\\
&=\eta(s_{\alpha',\iota(B)\setminus\cup_{\alpha\in\Delta_{[B]}}\iota_\alpha(\theta_\alpha(B))})\\
&=s_{\alpha',B}-s_{\alpha,'B}\sum_{\alpha\in\Delta_{[B]}}s_{\alpha,\theta_\alpha(B)}s_{\alpha,\theta_\alpha(B)}^*
\end{align*}
for $\alpha'\in\CL$ and $B\in\CI_{\alpha'}\cap\CS$.

A straightforward application of the Gauge-Invariant Uniqueness Theorem gives us that $\psi$ is injective. Since the image of $\psi$ is the $C^*$-subalgebra generated by 
\begin{multline*}
\bigl\{p_A-\sum_{\alpha\in\Delta_{[A]}}s_{\alpha,\theta_\alpha(A)}s_{\alpha,\theta_\alpha(A)}^*,\\ s_{\alpha',B}-s_{\alpha,'B}\sum_{\alpha\in\Delta_{[B]}}s_{\alpha,\theta_\alpha(B)}s_{\alpha,\theta_\alpha(B)}^*: A\in\CS,\ \alpha'\in\CL,\ B\in\CI_{\alpha'}\cap\CS\bigr\},
\end{multline*}
and $I_{(\CH,\CS)}$ is the ideal generated by the same set, it follows that the image of $\psi$ is a full hereditary $C^*$-subalgebra of $I_{(\CH,\CS)}$, and thus that $I_{(\CH,\CS)}$ and $C^*(\tilde{\CS},\CL,\tilde{\theta},\tilde{\CI}_\alpha\cap\tilde{\CS};\tilde{\CJ}\cap\tilde{\CS})$ are Morita equivalent.
\end{proof}

\begin{cor}
Let $(\CB,\CL,\theta,\CI_\alpha;\CJ)$ be a relative generalized Boolean dynamical system. Then every gauge-invariant ideal of $C^*(\CB,\CL,\theta,\CI_\alpha;\CJ)$ is Morita equivalent to the $C^*$-algebra of a relative generalized Boolean dynamical system.
\end{cor}

\section{$K_0$-liftability}\label{$K_0$-liftability} 

Let $(\CB,\CL,\theta,\CI_\alpha;\CJ)$ be a relative generalized Boolean dynamical system. We in this section prove that if  $(\CB,\CL,\theta)$ satisfies Condition $(K)$, then $C^*(\CB,\CL,\theta, \CI_\af;\CJ)$ is $K_0$-liftable.
We recall that a $C^*$-algebra is said to be {\it $K_0$-liftable} (\cite[Definition 3.1]{PR2007}) if every pair of ideals $I,J$ in $A$, the extension
$$\xymatrix{
0 \ar[r] & I  \ar[r]^{i}  & J \ar[r]^{\pi} & J/I \ar[r] & 0\\ 
}$$
has the property that $K_0(\pi): K_0(J) \to K_0(J/I)$ is surjective, or equivalently that 
$K_1(\iota): K_1(I) \to K_1(J)$ is injective. 
It is also known that $A$ is $K_0$-liftable if $K_1(\iota): K_1(I) \to K_1(A)$ is injective for every ideal $I$ of $A$.

\begin{thm}\label{K0-liftable} Let $(\CB,\CL,\theta, \CI_\af;\CJ)$ be a relative generalized Boolean dynamical system. If $(\CB,\CL,\theta)$ satisfies Condition $(K)$, then $C^*(\CB,\CL,\theta, \CI_\af;\CJ)$ is $K_0$-liftable.
\end{thm}

\begin{proof}
Choose a non-trivial ideal $I$ of $C^*(\CB,\CL,\theta, \CI_\af;\CJ)$. We shall show that the inclusion $\iota: I \to C^*(\CB,\CL,\theta, \CI_\af;\CJ)$ induces an injective map at the level of $K_1$-groups. Since $(\CB,\CL,\theta)$ satisfies Condition $(K)$, 
$I=I_{(\CH,\CS)}$ for some hereditary $\CJ$-saturated ideal $\CH$ of $\CB$ and an ideal $\CS$ of $\CB_\CH$  by \cite[Theorem 3.1]{CaK3} and \cite[Theorem 7.4]{CaK2}. Then, $I_{(\CH,\CS)}$ is Morita equivalent to 
$C^*(\widetilde{\CS},\CL,\widetilde{\theta},\widetilde{\CI}_\alpha\cap\widetilde{\CS};\widetilde{\CJ}\cap\widetilde{\CS})$
by Theorem \ref{gii}, which induces an isomorphism 
\[
\rho: K_1(C^*(\widetilde{\CS},\CL,\widetilde{\theta},\widetilde{\CI}_\alpha\cap\widetilde{\CS};\widetilde{\CJ}\cap\widetilde{\CS})) \to K_1(I_{(\CH,\CS)}).
\]
Applying Theorem \ref{generator of K1} to $(\widetilde{\CS},\CL,\widetilde{\theta},\tilde{\CI}_\alpha\cap\widetilde{\CS};\widetilde{\CJ}\cap\widetilde{\CS})$, we obtain a group isomorphism
\[
\chi_1': \ker ((1-\Phi)|_{\widetilde{\CJ}\cap\widetilde{\CS}} ) \to K_1(C^*(\widetilde{\CS},\CL,\widetilde{\theta},\widetilde{\CI}_\alpha\cap\widetilde{\CS};\widetilde{\CJ}\cap\widetilde{\CS})),
\]
where 
\[
(1-\Phi)|_{\widetilde{\CJ}\cap\widetilde{\CS}}: \operatorname{span}_\Z\{\chi_{\tilde{A}}: \tilde{A} \in \widetilde{\CJ}\cap\widetilde{\CS} \} \to \operatorname{span}_\Z\{\chi_{\tilde{A}}: \tilde{A} \in \widetilde{\CS}\}
\]
is defined by $\chi_{\tilde{A}} \mapsto \chi_{\tilde{A}} - \sum_{\af \in \triangle_{\tilde{A}}} \chi_{\widetilde{\theta}_{\af}(\tilde{A})}$ for $\tilde{A} \in \widetilde{\CJ}\cap\widetilde{\CS}$. Thus, we have a group isomorphism
\[
\rho \circ \chi_1': \ker ((1-\Phi)|_{\widetilde{\CJ}\cap\widetilde{\CS}}) \to K_1(I_{(\CH,\CS)}).
\]

Since $\widetilde{\CJ}\cap\widetilde{\CS} \subseteq \widetilde{\CJ}$, there is a natural inclusion 
\[
I_1: \ker ((1-\Phi)|_{\widetilde{\CJ}\cap\widetilde{\CS}}) \hookrightarrow \ker((1-\Phi)|_{\widetilde{\CJ}}),
\]
where 
\[
(1-\Phi)|_{\widetilde{\CJ}}: \operatorname{span}_\Z\{\chi_{\tilde{A}}: \tilde{A} \in \widetilde{\CJ}\} \to \operatorname{span}_\Z\{\chi_{\tilde{A}}: \tilde{A} \in \widetilde{\CB}\}
\]
is defined similarly. Let $\eta: C^*(\widetilde{\CB},\CL,\widetilde{\theta},\widetilde{\CI}_\alpha;\widetilde{\CJ}) \to C^*(\CB,\CL,\theta,\CI_\alpha;\CJ)$ be the isomorphism obtained from Theorem \ref{prop:one}, and let 
\[
\chi_1: \ker((1-\Phi)|_{\widetilde{\CJ}}) \rightarrow K_1(C^*(\widetilde{\CB},\CL,\widetilde{\theta},\widetilde{\CI}_\alpha;\widetilde{\CJ}))
\]
be the group isomorphism from Theorem \ref{generator of K1}. Then the following diagram commutes:
\[
\begin{tikzcd}
\ker((1 - \Phi)|_{\widetilde{\CJ} \cap \widetilde{\CS}}) 
  \arrow[r, "\chi_1'"] 
  \arrow[d, hookrightarrow, "I_1"'] 
  & K_1(C^*(\widetilde{\CS}, \CL, \widetilde{\theta}, \widetilde{\CI}_\alpha \cap \widetilde{\CS}; \widetilde{\CJ} \cap \widetilde{\CS})) 
  \arrow[r, "\rho"] 
  \arrow[d, hookrightarrow] 
  & K_1(I) 
  \arrow[d, "K_1(\iota)"] \\
\ker((1 - \Phi)|_{\widetilde{\CJ}}) 
  \arrow[r, "\chi_1"] 
  & K_1(C^*(\widetilde{\CB}, \CL, \widetilde{\theta}, \widetilde{\CI}_\alpha; \widetilde{\CJ})) 
  \arrow[r, "K_1(\eta)"] 
  & K_1(C^*(\CB, \CL, \theta, \CI_\alpha; \CJ)),
\end{tikzcd}
\] which can be succinctly expressed by the identity
\[
K_1(\iota) \circ (\rho \circ \chi_1') = (K_1(\eta) \circ \chi_1) \circ I_1.
\]
It thus follows that $K_1(\iota)$ is injective. Hence, $C^*(\CB,\CL,\theta, \CI_\af;\CJ)$ is $K_0$-liftable.

\end{proof}

\begin{cor}
Let $(\CB,\CL,\theta, \CI_\af;\CJ)$ be a relative generalized Boolean dynamical system such that $\CB$ and $\CL$ are countable. If  $C^*(\CB,\CL,\theta, \CI_\af;\CJ)$ is  purely infinite, then  $C^*(\CB,\CL,\theta, \CI_\af;\CJ)$ is of real rank zero, and, in this case,  $C^*(\CB,\CL,\theta, \CI_\af;\CJ)$ is of real rank zero if and only if $(\CB,\CL,\theta)$  satisfies Condition (K).
\end{cor}

\begin{proof}If  $C^*(\CB,\CL,\theta, \CI_\af;\CJ)$ is  purely infinite, then $(\CB, \CL,\theta)$ satisfies Condition (K) by \cite[Proposition 1.5 and Theorem 3.1]{CaK3}. Thus, $C^*(\CB,\CL,\theta, \CI_\af;\CJ)$ is $K_0$-liftable by Theorem \ref{K0-liftable}, and  the topological dimension of $C^*(\CB,\CL,\theta,\CI_\alpha;\CJ)$ is zero by \cite[Theorem 3.1]{CaK3}. Hence, by \cite[Theorem 4.2]{PR2007}, the real rank of  $C^*(\CB,\CL,\theta, \CI_\af;\CJ)$ is zero. 

The last statement follows immediately. 
\end{proof}

\begin{cor}\label{coro: Condition (K)}
Let $(\CB,\CL,\theta, \CI_\af;\CJ)$ be a relative generalized Boolean dynamical system such that $\CB$ and $\CL$ are countable. Then the following are equivalent.
\begin{enumerate}
\item $(\CB, \CL, \theta)$ satisfies Condition $(K)$.
\item $C^*(\CB,\CL,\theta, \CI_\af;\CJ) \otimes \CO_{2}$ has real rank zero.
\item $C^*(\CB,\CL,\theta, \CI_\af;\CJ) \otimes \CO_{\infty}$ has real rank zero.
\end{enumerate}
\end{cor}

\begin{proof} 
(1)$\iff$(2) follows by \cite[Corollary 4.3]{PR2007} and \cite[Theorem 3.1]{CaK3}.

(1)$\implies$(3): If $(\CB,\CL,\theta)$ satisfies Condition (K), then $C^*(\CB,\CL,\theta, \CI_\af;\CJ) $ has topological dimension zero and $K_0$-liftable. Thus $C^*(\CB,\CL,\theta, \CI_\af;\CJ) \otimes \CO_{\infty}$ has real rank zero by \cite[Corollary 4.3]{PR2007}. 

(3)$\implies$(1): If $C^*(\CB,\CL,\theta, \CI_\af;\CJ) \otimes \CO_{\infty}$ has real rank zero, then  $C^*(\CB,\CL,\theta, \CI_\af;\CJ) $ has topological dimension zero by \cite[Corollary 4.3]{PR2007}. Thus,  $(\CB,\CL,\theta)$ satisfies Condition (K) by \cite[Theorem 3.1]{CaK3}.
 \end{proof}

\end{document}